\documentclass[11pt,reqno]{amsart}

\usepackage[utf8]{inputenc}
\usepackage[T1]{fontenc}

\usepackage[margin=1in]{geometry}
\usepackage{amsmath,amsthm,amssymb}
\usepackage{mathtools}

\usepackage[hidelinks]{hyperref}
\usepackage{amsrefs}

\numberwithin{equation}{section}

\newtheorem{theorem}{Theorem}[section]
\newtheorem{lemma}[theorem]{Lemma}

\newtheorem{corollary}[theorem]{Corollary}

\theoremstyle{definition}
\newtheorem{definition}[theorem]{Definition}
\newtheorem{convention}[theorem]{Convention}

\theoremstyle{remark}
\newtheorem{remark}[theorem]{Remark}

\newcommand{\secref}[1]{Section~\ref{#1}}
\newcommand{\appref}[1]{Appendix~\ref{#1}}
\newcommand{\thmref}[1]{Theorem~\ref{#1}}
\newcommand{\lemref}[1]{Lemma~\ref{#1}}

\newcommand{\corref}[1]{Corollary~\ref{#1}}
\newcommand{\defref}[1]{Definition~\ref{#1}}

\newcommand{\remref}[1]{Remark~\ref{#1}}
\newcommand{\conref}[1]{Convention~\ref{#1}}

\title[Spectral-Operator Calculus I]{Spectral-Operator Calculus I:
Trace-Form Evaluators and Spectral Growth in the Self-Adjoint Setting}
\author{John Homer}
\date{December 2025}

\address{Independent Researcher, Longwood, Florida, USA}
\email{am9obi5ob21lcg-arxiv@yahoo.com}

\providecommand{\tightlist}{%
  \setlength{\itemsep}{0pt}%
  \setlength{\parskip}{0pt}%
}

\begin{document}

\begin{abstract}
We develop Spectral-Operator Calculus (SOC), an axiomatic calculus for
scalar evaluation of operator-generated spectral observables. This paper
(SOC-I) treats the self-adjoint setting, where observables are bounded
Borel transforms and locality is enforced via additivity across spectral
partitions. Under unitary invariance, extensivity on orthogonal sums,
projector-locality, and a dominated-convergence continuity condition, we
prove a rigidity theorem on a natural trace-class envelope: every
admissible evaluator agrees with a weighted trace of a single Borel
nondecreasing profile applied through the functional calculus. We then
introduce a spectral growth taxonomy based on eigenvalue counting
asymptotics and show that the polynomial growth regime is stable under
the basic constructions of the calculus. These results supply an
arithmetic-neutral analytic backbone for subsequent SOC parts and for
applications to concrete spectral models. A companion part treats the
sectorial/holomorphic setting, where locality is formulated on log-scale
via scale-band decompositions and positive kernels rather than spectral
projections.
\end{abstract}

\maketitle

\emph{Series note.} The SOC series develops an evaluator calculus for
spectral observables generated by functional calculi. The common spine
is an axiomatics for scalar evaluators (unitary invariance, extensivity,
localization under coarse-graining, dominated-continuity, and
normalization), together with rigidity and stability consequences that
are designed to persist under admissible limiting schemes.

SOC-I develops the self-adjoint/PVM-localized model case. It fixes the
evaluator axioms, proves a trace-form rigidity theorem on a natural
trace-class envelope, and introduces a spectral growth taxonomy and
continuity framework for scalar evaluations built from spectra. A
companion part develops the sectorial/holomorphic model case (Haase-type
functional calculus; see Haase \cite{Haase2020}), where the localization
gate is implemented on log-scale by scale-band decompositions and kernel
positivity rather than by spectral projections.

Later parts build on these model cases to study admissible kernels and
cutoff schemes, analytic continuation procedures, and boundary and
reconstruction questions.

\section{Introduction}\label{introduction}

\label{sec:introduction}

The basic objects in this paper are self-adjoint operators, their
spectral measures, and bounded Borel transforms. Concretely, we work on
a fixed complex separable Hilbert space \(H\), regard self-adjoint
operators \(D\) on \(H\) as models, and treat bounded Borel functions
\(f:\mathbb{R}\to\mathbb{C}\) as generating bounded operators \(f(D)\)
through the functional calculus provided by the spectral theorem. These
operators \(f(D)\) represent the spectral geometries on which our
scalar-valued functionals act. In this sense, all primitive data in the
calculus are operators and spectra, not bare scalar sequences.

At the level of structure, SOC is organized around three interacting
gates: (i) a localization rule that specifies how evaluator values
decompose under coarse-graining; (ii) a dominated-continuity rule that
legalizes passage to limits inside the evaluator; and (iii) a
normalization/nontriviality rule that excludes degenerate assignments.
The operator class and functional calculus supply the observable family
and the available coarse-grainings. In the self-adjoint setting of this
part, coarse-graining is spectral: we partition the spectrum and demand
additivity across the corresponding spectral projections. In the
sectorial setting, spectral projections are unavailable in general; SOC
instead uses scale-band decompositions \(t \mapsto \psi(tA)\) and
controls interference through a positive kernel on log-scale. This paper
develops the PVM-localized model case and proves that the gates force
trace-form rigidity on the trace-class envelope.

\subsection{Main results and
contributions}\label{main-results-and-contributions}

The central structural object of the paper is an evaluator \[
\mathcal{E} : \mathcal{C} \to [0,\infty]
\] defined on a class \(\mathcal{C}\) of bounded spectral geometries.
The class \(\mathcal{C}\) is required to be stable under bounded Borel
transforms and under the operations needed later for examples and
continuity statements. The evaluator \(\mathcal{E}\) is required to
satisfy natural axioms: unitary invariance, locality on spectral
projections, extensivity on orthogonal sums, a dominated-convergence
continuity property, and a mild growth condition that prevents
trivialization (see Axiom A5 in \secref{subsec:axioms}). The guiding
question is whether such an abstract evaluator is forced to be of trace
form, and, if so, how rigid that representation is.

Our first main result is \thmref{thm:trace-form}, a trace-form
representation theorem for evaluators. Under the standing hypotheses on
\(\mathcal{C}\) and the evaluator axioms, there exist a Borel-measurable
nondecreasing function \(h:\mathbb{R}\to\mathbb{R}\) and a normalization
constant \(c_{\mathcal{E}}\ge 0\) such that \(\mathcal{E}(f(D))\) is
given by the trace of \(h(f(D))\) whenever this operator is trace class.
In particular, if we restrict \(\mathcal{E}\) to the trace-class
envelope \[
\mathcal{C}_1 := \mathcal{C} \cap \mathcal{S}_1(H),
\] then every admissible value \(\mathcal{E}(X)\) with
\(X\in\mathcal{C}_1\) is completely determined by the spectral geometry
of \(X\) together with the fixed profile \(h\) and the overall scale
\(c_{\mathcal{E}}\). In this sense the evaluator has no additional
degrees of freedom once the profile and normalization are fixed.

The second main result is a spectral growth taxonomy for self-adjoint
operators. We work with counting functions \(N_D(\Lambda)\) and related
spectral distributions, and use their asymptotic behaviour as
\(\Lambda\to\infty\) to organize operators into growth classes. In
particular, we single out a polynomial growth regime that captures the
standard elliptic and Laplace-type models and show in
\thmref{thm:growth-closure} that this class is stable under the basic
constructions of the calculus, including monotone regularly varying
functional calculus, finite orthogonal sums, and tensor products. This
growth framework is then used in examples and in the continuity theory
for evaluators in \secref{sec:examples}.

A further theme of the paper is continuity under natural limiting
procedures. We study approximations of spectral geometries by
truncations and cutoffs and show that the evaluator values behave well
under dominated limits. In particular, we prove that monotone and
dominated approximations by bounded transforms do not introduce spurious
boundary contributions in the trace-form representation.

Conceptually, the contribution of this paper is twofold. First, at the
level of scalar evaluation, the axioms force a rigidity phenomenon on
the trace-class envelope: any evaluator satisfying the standing
structural conditions must be given by a single nondecreasing profile
applied through the functional calculus and traced, so that scalar
values carry no independent structure beyond what is encoded in the
spectra and the chosen profile. Second, on the geometric side, the
spectral growth classification and closure results provide a stable
regime in which the spectral calculus and the trace-form representation
interact well under the operations needed in later parts of the series.

Throughout, the presentation is self-contained and arithmetic-neutral.
All statements are made at the level of operators, spectra, and traces,
without appealing to arithmetic or number-theoretic input. The intent is
that the evaluator representation, growth taxonomy, and continuity
framework developed here can be reused across different spectral models
and applications without revisiting these analytic foundations.

\subsection{Relation to prior work}\label{relation-to-prior-work}

There is a substantial literature on trace ideals, unitarily invariant
norms, and operator inequalities, with classical references including
Reed-Simon \cite{ReedSimon1972}, Simon's monograph on trace ideals
\cite{Simon2005}, and Bhatia's text on matrix analysis
\cite{Bhatia1997}. In that setting one typically starts from a given
trace and studies the properties of associated norms and functionals on
operator ideals. The perspective adopted here is different: we start
from an abstract evaluator \(\mathcal{E}\) defined on a class
\(\mathcal{C}\) of spectral geometries, subject only to unitary
invariance, projector-locality, dominated continuity, and a growth
bound, and obtain a universal trace-form representation for such
evaluators on a trace-class envelope. In this sense
\thmref{thm:trace-form} can be viewed as a representation theorem for
evaluator functionals, rather than as a study of a fixed underlying
trace.

Classical uniqueness theorems for symmetric gauge functionals work in
the opposite direction: they assume the trace or norm structure is given
a priori and then classify functionals compatible with it, whereas here
trace form itself is derived as a consequence of the evaluator axioms,
without presupposing any trace or ideal structure.

On the functional-calculus side, SOC is modular: it takes as input an
operator class equipped with a functional calculus and a compatible
notion of coarse-graining, and it studies which scalar evaluators are
forced by invariance, locality, and dominated-limit legality. SOC-I
instantiates this program in the self-adjoint/Borel setting, where the
spectral theorem supplies both the calculus and the canonical
coarse-grainings (spectral partitions via the PVM). Haase's framework
for sectorial operators \cite{Haase2020} supplies the corresponding
input for a second SOC model case based on holomorphic functional
calculus, where PVM locality is replaced by scale-band decompositions
and kernel positivity on log-scale. The novelty in SOC is not the
existence of these calculi, but the rigidity and stability constraints
on evaluators that follow once a localization structure is fixed.

On the growth side, the taxonomy introduced here is influenced by the
language of regular variation and spectral asymptotics.
Counting-function techniques and regularly varying functions are
standard tools in the analysis of eigenvalue distributions, as in the
work of Bingham-Goldie-Teugels \cite{BGT1987} and in more recent
treatments of spectral densities and hypertraces in noncommutative
geometry \cite{CiprianiSauvageot2023}, as well as in explicit
counting-function studies for fractal sprays and related models
\cite{KombrinkSchmidt2023}. Our contribution is to formulate a growth
classification directly in terms of counting functions and to prove
stability of the polynomial class under the structural operations
required by the evaluator calculus, such as functional calculus,
orthogonal sums, and tensor products, without committing to a specific
geometric model.

This paper should therefore be read as an operator-theoretic calculus
for evaluators and spectral growth, rather than as a contribution to the
fine structure of trace ideals or spectral asymptotics. It does not aim
to optimize asymptotic constants or to exhaust the classical theory;
instead, it isolates a structural setting in which evaluators, spectral
geometries, and growth conditions fit together coherently, and in which
the main representation and continuity results can be applied across
different models without further analytic reconstruction. The results of
SOC-I are intended to be reusable: they provide an analytic backbone for
later work on kernels, cutoff schemes, and flow constructions on
spectral data, while the present part can be read independently of any
later developments in the series.

\subsection{Organization of the paper}\label{organization-of-the-paper}

\secref{sec:preliminaries} collects the spectral calculus, evaluator
notation, and the standing assumptions on the stable class
\(\mathcal{C}\). \secref{sec:axioms-main} states the evaluator axioms
and the two main results: the trace-form representation theorem and the
spectral growth taxonomy. \secref{sec:proofs} proves these results,
delegating measure-theoretic details to \appref{app:trace-form} and
counting-function estimates to \appref{app:growth-stability}.
\secref{sec:examples} illustrates the continuity theory for evaluators
and gives examples of spectral geometries in the growth classes
introduced in \secref{sec:axioms-main}.

\section{Preliminaries}\label{preliminaries}

\label{sec:preliminaries}

Throughout, \(H\) denotes a complex separable Hilbert space with inner
product \(\langle \cdot,\cdot\rangle\) linear in the second argument and
norm \(\|x\| = \sqrt{\langle x,x\rangle}\). All operators are densely
defined and self-adjoint unless explicitly stated otherwise. For a
self-adjoint operator \(D\) on \(H\), we write \(\sigma(D)\) for its
spectrum and \(E_D\) for the associated projection-valued spectral
measure on the Borel subsets of \(\mathbb{R}\).

\subsection{Spectral measures and functional
calculus}\label{spectral-measures-and-functional-calculus}

\label{subsec:spectral-calculus}

Let \(D\) be self-adjoint on \(H\). By the spectral theorem, there
exists a unique projection-valued measure \(E_D\) on the Borel
\(\sigma\)-algebra of \(\mathbb{R}\) such that
\begin{equation}\label{eq:spectral-decomposition}
D = \int_{\mathbb{R}} \lambda\, dE_D(\lambda)
\end{equation} in the sense of strong operator convergence on the domain
of \(D\); see Reed-Simon \cite{ReedSimon1972}*\{Theorem VIII.3\}. For
any bounded Borel function \(f:\mathbb{R}\to\mathbb{C}\), the bounded
operator \begin{equation}\label{eq:spectral-transform}
f(D) := \int_{\mathbb{R}} f(\lambda)\, dE_D(\lambda)
\end{equation} is defined by spectral calculus. We will refer to such
operators \(f(D)\), with \(f\) bounded Borel, as spectral geometries
associated with \(D\). In particular, the basic objects of the calculus
will be self-adjoint operators \(D\), their spectral measures \(E_D\),
and the bounded spectral geometries \(f(D)\) built from them. If \(f\)
is real-valued, then \(f(D)\) is self-adjoint. The assignment
\(f\mapsto f(D)\) is a unital \(*\)-homomorphism from the algebra of
bounded Borel functions on \(\mathbb{R}\) into the bounded operators on
\(H\).

For vectors \(x,y\in H\), the scalar measure \[
\mu_{x,y}(B) := \langle E_D(B)x,y\rangle, \qquad B\subset\mathbb{R}\ \text{Borel},
\] is complex-valued and countably additive; when \(x=y\) we write
\(\mu_x := \mu_{x,x}\), a finite positive measure. For bounded Borel
\(f\), \[
\langle f(D)x,y\rangle = \int_{\mathbb{R}} f(\lambda)\, d\mu_{x,y}(\lambda).
\]

When \(D\) has pure-point spectrum with finite multiplicities, we write
\(\{\lambda_n(D)\}_{n\ge1}\) for its eigenvalues (repeated according to
multiplicity) and choose an orthonormal basis of eigenvectors. Then
\(D\) is unitarily equivalent to a diagonal operator on \(\ell^2\), and
for bounded Borel \(f\), \[
f(D) \sim \operatorname{diag}\big(f(\lambda_1(D)), f(\lambda_2(D)),\ldots\big).
\]

\begin{remark}[Borel functional calculus]\label{rem:borel-functional-calculus}
In what follows, whenever a Borel-measurable function \(h:\mathbb{R}\to\mathbb{R}\) is applied to a self-adjoint operator \(D\), the operator \(h(D)\) is understood in the sense of the Borel functional calculus described above. In particular, \thmref{thm:trace-form} will invoke \(h(f(D))\) with \(f\) bounded Borel and \(h\) Borel-measurable and nondecreasing; no continuity of \(h\) beyond Borel measurability is required to define \(h(f(D))\).
\end{remark}

\subsection{Trace class operators and
traces}\label{trace-class-operators-and-traces}

Let \(\mathcal{S}_1(H)\) denote the ideal of trace-class operators on
\(H\). For \(T \in \mathcal{S}_1(H)\), its trace is defined by \[
\operatorname{Tr}(T) = \sum_{n\ge1} \langle T e_n, e_n\rangle,
\] where \(\{e_n\}\) is any orthonormal basis of \(H\); the value is
independent of the basis. If \(T\) is positive trace class,
\(\operatorname{Tr}(T)\) coincides with the sum of its eigenvalues
counted with multiplicity.

For a self-adjoint \(D\) and bounded Borel \(f\), the operator \(f(D)\)
need not be trace class. Whenever \(f(D) \in \mathcal{S}_1(H)\), the
spectral theorem gives \begin{equation}\label{eq:trace-integral}
\operatorname{Tr}(f(D)) = \int_{\mathbb{R}} f(\lambda)\, d\nu_D(\lambda)
\end{equation} where the spectral trace measure \(\nu_D\) is defined by
\[
\nu_D(B) = \operatorname{Tr}(E_D(B)), \qquad B \subset \mathbb{R}\ \text{Borel},
\] whenever this expression is finite. In many situations, \(\nu_D\) is
infinite on unbounded sets but finite on compact sets; in that case the
above formula is interpreted under explicit integrability or cutoff
hypotheses stated later.

A regularized trace on an operator \(A\) is any assignment
\(A \mapsto \operatorname{Tr}_{\mathrm{reg}}(A)\) defined on a class of
operators that extends \(\operatorname{Tr}\) on \(\mathcal{S}_1(H)\) and
is compatible with spectral calculus. In this paper we only use
spectral-cutoff regularization. For each self-adjoint \(D\) we fix a
family of bounded Borel functions
\(\chi_\Lambda : \mathbb{R} \to [0,1]\) with
\(\chi_\Lambda(\lambda) = 1\) for \(|\lambda| \le \Lambda\),
\(\chi_\Lambda(\lambda) = 0\) for \(|\lambda| \ge 2\Lambda\), and
\(\chi_\Lambda(\lambda) \to 1\) pointwise as \(\Lambda \to \infty\). Let
\[
T_\Lambda(D) = \chi_\Lambda(D)
\] be the corresponding spectral cutoff operators. Whenever the limit
exists, we define the spectral-cutoff trace of \(f(D)\) by \[
\operatorname{Tr}_{\mathrm{cut}}(f(D)) = \lim_{\Lambda \to \infty} \operatorname{Tr}\big(T_\Lambda(D)\,f(D)\big).
\] Whenever we speak of a regularized trace or a compatible regularized
trace in the sequel, we mean this spectral-cutoff trace
\(\operatorname{Tr}_{\mathrm{cut}}\), taken with respect to the fixed
family \(\{T_\Lambda(D)\}\). No other regularization schemes are used in
this paper.

\subsection{Counting functions and
growth}\label{counting-functions-and-growth}

When \(D\) has discrete spectrum with finite multiplicities accumulating
only at \(\pm\infty\), its eigenvalues can be numbered as
\(\{\lambda_n(D)\}_{n\ge1}\), and the associated counting function is \[
N_D(\Lambda) := \#\{n\ge1 : |\lambda_n(D)| \le \Lambda\}, \qquad \Lambda\ge0,
\] where each eigenvalue is counted with multiplicity.

\begin{convention}[Counting-function regularity]\label{conv:counting-function-regularity}
Whenever we work with a counting function \(N_D\), we assume that \(D\) has discrete spectrum with finite multiplicities, so that
\[
N_D(\Lambda) < \infty \qquad \text{for all } \Lambda \ge 0,
\]
and that \(N_D(\Lambda)\) is eventually nondecreasing as \(\Lambda\to\infty\); that is, there exists \(\Lambda_0 \ge 0\) such that
\[
\Lambda_1 \le \Lambda_2,\ \Lambda_1,\Lambda_2 \ge \Lambda_0
\quad\Longrightarrow\quad
N_D(\Lambda_1) \le N_D(\Lambda_2).
\]

The asymptotic behavior of \(N_D(\Lambda)\) as \(\Lambda\to\infty\) defines the growth class of \(D\). Typical examples include polynomial growth
\[
N_D(\Lambda) \sim C\, \Lambda^d, \qquad \Lambda\to\infty,
\]
stretched-exponential behavior
\[
N_D(\Lambda) \sim \exp\big(c\, \Lambda^\alpha\big), \qquad c>0,\ 0<\alpha<1,
\]
and logarithmic or slower growth. These classes will be formalized in \secref{sec:axioms-main} and shown to be stable under the structural operations needed for the spectral calculus.
\end{convention}

\begin{remark}[Counting functions and spectral measures]\label{rem:counting-vs-spectral}
Three related objects will appear repeatedly:

- The \emph{counting function} \(N_D(\Lambda)\) counts eigenvalues of a discrete-spectrum operator \(D\) up to a cutoff \(\Lambda\).
- The \emph{trace spectral measure} \(\nu_D\) of \(D\) is defined (when finite) by
  \[
  \nu_D(B) := \operatorname{Tr}(E_D(B)), \qquad B \subset \mathbb{R}\ \text{Borel},
  \]
  and encodes multiplicities via the spectral projections \(E_D(B)\).
- Given an evaluator \(\mathcal{E}\), the \emph{evaluator measure} \(\mu_D\) is defined (when \(E_D(B)\in\mathcal{C}\)) by
  \[
  \mu_D(B) := \mathcal{E}(E_D(B)).
  \]

In the growth taxonomy of \secref{subsec:growth-classes} and \appref{app:growth-stability} we work primarily with \(N_D\). In \appref{app:trace-form} and \secref{subsec:proof-trace-form} we work with the scalar measures \(\nu_D\) and \(\mu_D\) attached to the spectral measure \(E_D\); there we use the Radon-Nikodym theorem to express \(\mu_D\) as a density with respect to \(\nu_D\) (see Remark\nobreakspace\ref{rem:RN-density-Borel}). These three objects carry complementary information and will not be interchanged: \(N_D\) is a counting function, \(\nu_D\) encodes the ordinary trace, and \(\mu_D\) encodes the evaluator.
\end{remark}

\subsection{Spectral preorder and
majorization}\label{spectral-preorder-and-majorization}

Let \(X\) and \(Y\) be self-adjoint operators on \(H\) with discrete
spectra and finite multiplicities. Write \(\{\lambda_n(X)\}_{n\ge1}\)
and \(\{\lambda_n(Y)\}_{n\ge1}\) for their eigenvalues, counted with
multiplicity and arranged in nonincreasing order of absolute value: \[
|\lambda_1(X)| \ge |\lambda_2(X)| \ge \cdots, \qquad |\lambda_1(Y)| \ge |\lambda_2(Y)| \ge \cdots.
\] We write \[
\lambda_n^{\downarrow}(X) := \text{the \(n\)-th eigenvalue of \(X\) in this ordering}.
\]

\begin{definition}[Spectral preorder]\label{def:spectral-preorder}
For self-adjoint \(X,Y\) with discrete spectra as above, we write
\[
X \preceq Y
\]
if for every \(k\ge1\),
\[
\sum_{n=1}^k \lambda_n^{\downarrow}(X) \;\le\; \sum_{n=1}^k \lambda_n^{\downarrow}(Y).
\]
This preorder is invariant under unitary conjugation and coincides with the usual notion of spectral majorization; see, for example, Fan \cite{Fan1951} and Kosaki \cite{Kosaki1982} for background and equivalent formulations (e.g. in terms of Ky Fan norms and convex trace functionals).

When \(X\) and \(Y\) are positive and trace class, the spectral preorder \(\preceq\) interacts naturally with unitarily invariant trace functionals. In particular, if \(X\preceq Y\) and \(\Phi:[0,\infty)\to[0,\infty)\) is convex and nondecreasing with \(\Phi(0)=0\), then
\[
\operatorname{Tr}\big(\Phi(X)\big) \;\le\; \operatorname{Tr}\big(\Phi(Y)\big),
\]
whenever both traces are finite. This type of inequality underlies the spectral monotonicity properties used later in the growth taxonomy and evaluator arguments.
\end{definition}

\begin{remark}[Scope of the preorder in this paper]\label{rem:preorder-scope}
In this paper we use the spectral preorder \(\preceq\) only for operators with discrete spectra and finite multiplicities, as in \defref{def:spectral-preorder}. All subsequent appearances of \(\preceq\) refer to that setting. Extensions of spectral majorization to more general self-adjoint operators via strong resolvent limits are well developed in the literature (see, for example, Bhatia \cite{Bhatia1997}*{Chapter III}), but we do not rely on any such extension here, and no properties of \(\preceq\) beyond the discrete-spectrum case are needed for the results of \secref{sec:axioms-main}, \secref{sec:proofs}, and \secref{sec:examples}.
\end{remark}

\subsection{Evaluators and structural
notation}\label{evaluators-and-structural-notation}

A spectral geometry (or spectral transform) of a self-adjoint operator
\(D\) is any bounded Borel transform \(f(D)\) obtained from \(D\) via
the functional calculus of \secref{subsec:spectral-calculus}. Thus \[
f(D) = \int_{\mathbb{R}} f(\lambda)\, dE_D(\lambda)
\] for some bounded Borel function \(f:\mathbb{R}\to\mathbb{C}\), and
\(f(D)\) is a bounded operator on \(H\). When \(f\) is real-valued,
\(f(D)\) is self-adjoint.

We will work with a fixed \emph{stable class} \(\mathcal{C}\) of bounded
spectral geometries. By this we mean a family \(\mathcal{C}\) of
self-adjoint bounded operators on \(H\) satisfying:

\begin{enumerate}
\def\labelenumi{\arabic{enumi}.}
\tightlist
\item
  If \(X\in\mathcal{C}\) and \(f:\mathbb{R}\to\mathbb{R}\) is bounded
  and Borel measurable, then \(f(X)\in\mathcal{C}\). (Since \(f\) is
  real-valued and \(X\) is self-adjoint, the spectral calculus of
  \secref{subsec:spectral-calculus} ensures that \(f(X)\) is again
  self-adjoint.)
\item
  If \(X,Y\in\mathcal{C}\) act on orthogonal subspaces and \(X\oplus Y\)
  denotes their block-diagonal sum, then \(X\oplus Y\in\mathcal{C}\).
\item
  If \(X_n\in\mathcal{C}\) is a monotone sequence of self-adjoint
  operators converging strongly to \(X\), then \(X\in\mathcal{C}\).
\end{enumerate}

In particular, \(\mathcal{C}\) is closed under bounded Borel functional
calculus, finite orthogonal sums, and strong limits of monotone
sequences whenever the limit exists. These closure properties will be
part of the standing assumptions on the stable class throughout the
paper.

An \emph{evaluator} on \(\mathcal{C}\) is a map \[
\mathcal{E}:\mathcal{C}\to[0,\infty]
\] assigning a nonnegative extended real number to each bounded spectral
geometry in \(\mathcal{C}\). The structural axioms imposed on evaluators
- unitary invariance, extensivity on orthogonal sums,
projector-locality, dominated continuity, and a normalization/growth
bound - will be stated and fixed in \secref{subsec:axioms} as Axioms
A1-A5. Until then, \(\mathcal{E}\) is only used as notation for a
generic scalar-valued functional on \(\mathcal{C}\).

These conventions complete the preparatory framework for the evaluator
calculus: \secref{sec:axioms-main} will formalize the axioms satisfied
by \(\mathcal{E}\) and state the main representation and growth results
under those assumptions.

\section{Evaluator axioms and main
results}\label{evaluator-axioms-and-main-results}

\label{sec:axioms-main}

In this section we fix the structural hypotheses on the evaluator and on
the class \(\mathcal{C}\) of spectral geometries on which it acts, and
we state the two main theorems of the paper. The first result is a
trace-form representation theorem: any evaluator satisfying the axioms
below is of trace form on a natural trace-class envelope
\(\mathcal{C}_1\). The second result is a spectral growth taxonomy for
self-adjoint operators, together with basic closure properties of the
corresponding growth classes.

\subsection{Standing assumptions and
axioms}\label{standing-assumptions-and-axioms}

\label{subsec:axioms}

In this section we fix the structural framework for the evaluator
calculus used throughout the paper. All later results in SOC-I are
formulated under these standing assumptions.

We work on a fixed complex separable Hilbert space \(H\). All operators
are understood to act on \(H\), and all spectral measures and functional
calculi are taken with respect to this space. For a self-adjoint
operator \(D\) with spectral measure \(E_D(\cdot)\) and any bounded
Borel function \(f:\mathbb{R}\to\mathbb{R}\), the bounded transform \[
f(D) = \int_{\mathbb{R}} f(\lambda)\, dE_D(\lambda)
\] is defined via the spectral calculus of
\secref{subsec:spectral-calculus}.

We consider a class \(\mathcal{C}\) of bounded spectral geometries built
from self-adjoint operators and bounded Borel functions, in the
following sense.

\begin{itemize}
\tightlist
\item
  A typical element of \(\mathcal{C}\) has the form \(f(D)\), where
  \(D\) is a self-adjoint operator on \(H\) and
  \(f:\mathbb{R}\to\mathbb{R}\) is bounded and Borel measurable, with \[
  f(D) = \int_{\mathbb{R}} f(\lambda)\, dE_D(\lambda)
  \] defined via the spectral calculus of
  \secref{subsec:spectral-calculus}.
\item
  In particular, every \(X\in\mathcal{C}\) is a bounded self-adjoint
  operator on \(H\).
\item
  The class \(\mathcal{C}\) is assumed to be \emph{stable} under the
  operations needed in the sequel: if \(X\in\mathcal{C}\) and
  \(g:\mathbb{R}\to\mathbb{R}\) is bounded Borel, then
  \(g(X)\in\mathcal{C}\); if \(X,Y\in\mathcal{C}\) act on orthogonal
  subspaces, then \(X\oplus Y\in\mathcal{C}\); and if
  \(X_n\in\mathcal{C}\) is a monotone sequence converging strongly to
  \(X\), then \(X\in\mathcal{C}\).
\end{itemize}

We write \[
\mathcal{C}_1 := \mathcal{C}\cap\mathcal{S}_1(H)
\] for the trace-class envelope of \(\mathcal{C}\), where
\(\mathcal{S}_1(H)\) denotes the trace-class operators on \(H\). The
trace-form representation theorem will show that on \(\mathcal{C}_1\)
any evaluator satisfying the axioms below must be of trace form. Outside
\(\mathcal{C}_1\) we do not assert that \(\mathcal{E}\) has trace form.

We equip \(\mathcal{C}\) with an evaluator \[
\mathcal{E}:\mathcal{C}\to[0,\infty]
\] satisfying the following axioms.

\emph{Axiom A1 (Unitary invariance).}\\
If \(U\) is unitary and \(X\in\mathcal{C}\), then \[
\mathcal{E}(U^* X U) = \mathcal{E}(X).
\]

\emph{Axiom A2 (Extensivity on orthogonal sums).}\\
If \(X,Y\in\mathcal{C}\) act on orthogonal subspaces and if
\(X\oplus Y\in\mathcal{C}\) denotes their block-diagonal sum, then \[
\mathcal{E}(X\oplus Y) = \mathcal{E}(X) + \mathcal{E}(Y).
\]

\emph{Axiom A3 (Projector locality).}\\
If \(X\in\mathcal{C}\) has a finite spectral decomposition \[
X = \sum_{j=1}^m \lambda_j P_j
\] with pairwise orthogonal spectral projections \(P_j\) and real
scalars \(\lambda_j\), then \[
\mathcal{E}(X) = \sum_{j=1}^m \mathcal{E}(\lambda_j P_j),
\] whenever the right-hand side makes sense in \([0,\infty]\). In
particular, the value of \(X\) is determined by its values on spectral
projector blocks.

\emph{Axiom A4 (Dominated continuity in the strong operator topology).}\\
Let \(X_n,X,Y\in\mathcal{C}\) be positive operators such that

\begin{itemize}
\tightlist
\item
  \(X_n \to X\) strongly, and
\item
  \(0 \le X_n \le Y\) for all \(n\),
\end{itemize}

with \(\mathcal{E}(Y) < \infty\). Then \[
\mathcal{E}(X_n) \to \mathcal{E}(X).
\]

Here \(X_n \to X\) strongly means convergence in the strong operator
topology. In applications we will use Axiom A4 primarily for bounded
spectral geometries \(f(D)\) with \(f\) bounded Borel, where it
justifies the monotone limit and dominated-convergence arguments in the
proof of the trace-form representation and in the continuity examples of
\secref{sec:examples}.

\emph{Axiom A5 (Normalization and growth bound).}\\
We assume that the identity operator \(I\) belongs to \(\mathcal{C}\),
that \[
0 < \mathcal{E}(I) < \infty,
\] and that for every scalar \(t\ge0\) the scalar multiple \(tI\) lies
in \(\mathcal{C}\) and is evaluated by \[
\mathcal{E}(tI) = t\,\mathcal{E}(I).
\] Thus the evaluator is nontrivial and finite on the positive scalar
multiples of the identity, and these operators provide a harmless growth
bound for dominated-convergence arguments (for example, when comparing a
positive \(X\in\mathcal{C}\) with \(Y=\|X\|_{\mathrm{op}}\,I\)).

These axioms fix the structural setting for the trace-form
representation theorem and the growth taxonomy developed in the rest of
the paper. In particular, they encode the requirements that any lawful
evaluator must satisfy in order to behave well under unitary
conjugation, orthogonal decomposition, spectral partitions, and
dominated limits. In terms of the three-gate structure of
\secref{sec:introduction}: Axioms A1 and A3 implement the localization
gate; Axiom A4 implements dominated-continuity; and Axiom A5 implements
normalization, while Axiom A2 (extensivity) is a structural
compatibility condition. Under Axioms A1-A5, Theorem 3.1 will show that
on \(\mathcal{C}_1\) the evaluator \(\mathcal{E}\) is necessarily of
trace form with a single Borel nondecreasing profile and a global
normalization constant, and \secref{subsec:growth-classes} will
introduce the growth classes compatible with this framework.

\subsection{Trace-form representation of
evaluators}\label{trace-form-representation-of-evaluators}

We now state the first main theorem. Under the standing assumptions of
\secref{subsec:axioms}, any evaluator \(\mathcal{E}\) on a stable class
\(\mathcal{C}\) of spectral geometries that satisfies Axioms A1-A5 is
necessarily of trace form on the trace-class envelope \(\mathcal{C}_1\).

\begin{theorem}[Trace-form representation]\label{thm:trace-form}
Let \(\mathcal{C}\) and \(\mathcal{E}\) satisfy the standing assumptions of \secref{subsec:axioms}, and write
\[
\mathcal{C}_1 := \mathcal{C}\cap\mathcal{S}_1(H)
\]
for the trace-class envelope of \(\mathcal{C}\). Then there exist

- a Borel-measurable, nondecreasing function \(h:\mathbb{R}\to\mathbb{R}\) with \(h(0)=0\), and  
- a constant \(c_{\mathcal{E}}>0\),

such that the following holds.

1. For every positive operator \(X \in \mathcal{C}_1\),
   \[
   \mathcal{E}(X) = c_{\mathcal{E}}\;\operatorname{Tr}\big(h(X)\big).
   \]

2. More generally, for every self-adjoint operator \(D\) on \(H\) and every bounded Borel function \(f:\mathbb{R}\to\mathbb{R}\) such that \(f(D)\in\mathcal{C}_1\),
   \[
   \mathcal{E}\big(f(D)\big) = c_{\mathcal{E}}\;\operatorname{Tr}\big(h(f(D))\big).
   \]

Here \(h(X)\) and \(h(f(D))\) are defined by the Borel functional calculus of \secref{subsec:spectral-calculus}, and the traces are finite for all \(X\in\mathcal{C}_1\) by construction of \(\mathcal{C}_1\).

Moreover, if \(\tilde h:\mathbb{R}\to\mathbb{R}\) is another Borel-measurable, nondecreasing function and \(\tilde c_{\mathcal{E}}>0\) is a constant such that
\[
\mathcal{E}(X) = \tilde c_{\mathcal{E}}\;\operatorname{Tr}\big(\tilde h(X)\big)
\qquad\text{for all } X\in\mathcal{C}_1,
\]
then there exist \(a>0\) and \(b\in\mathbb{R}\) such that
\[
\tilde h = a\,h + b
\quad\text{and}\quad
\tilde c_{\mathcal{E}} = c_{\mathcal{E}}/a.
\]
In particular, on \(\mathcal{C}_1\) the pair \((h,c_{\mathcal{E}})\) is unique up to multiplication of \(h\) by a positive scalar and addition of a constant.
\end{theorem}

\begin{proof}
See \secref{subsec:proof-trace-form} and \appref{app:trace-form}.
\end{proof}

\begin{remark}[Scope of \thmref{thm:trace-form}]\label{rem:trace-form-scope}
\thmref{thm:trace-form} applies only to the trace-class envelope \(\mathcal{C}_1 := \mathcal{C}\cap\mathcal{S}_1(H)\) of the evaluator domain. In particular, if \(X\in\mathcal{C}\) is not trace class (for example, the identity on an infinite-dimensional space or a diagonal operator with slowly decaying eigenvalues), the theorem does not assert that \(\mathcal{E}(X)\) can be written as a trace \(\operatorname{Tr}(h(X))\) for any profile \(h\). Non-trace-class spectral geometries enter the present part only through the growth taxonomy of \secref{subsec:growth-classes} and the continuity theory of \secref{sec:proofs}; any extension of evaluators to them via regularization lies outside the scope of this paper. The dominated-continuity hypothesis A4 is tailored to this setting and will be used in \secref{sec:proofs} and \secref{subsec:dominated-convergence} for bounded spectral geometries \(f(D)\) with \(f\) bounded Borel, where it justifies the monotone limit and dominated-convergence arguments in the proof of the representation theorem and in the continuity examples.
\end{remark}

\begin{remark}[Evaluator rigidity on \(\mathcal{C}_1\)]\label{rem:evaluator-rigidity}
\thmref{thm:trace-form} shows that, on the trace-class envelope \(\mathcal{C}_1\), the evaluator \(\mathcal{E}\) has no additional degrees of freedom beyond the choice of the profile \(h\) and the normalization constant \(c_{\mathcal{E}}\). Once \(\mathcal{C}\), the axioms A1-A5, and the overall scale are fixed, every admissible value \(\mathcal{E}(X)\) with \(X\in\mathcal{C}_1\) is given by \(\operatorname{Tr}(h(X))\) up to the constant factor \(c_{\mathcal{E}}\). In particular, any other evaluator \(\tilde{\mathcal{E}}\) satisfying the same axioms and finite on \(\mathcal{C}_1\) differs only by a change of normalization.
\end{remark}

\subsection{Spectral growth classes and
closure}\label{spectral-growth-classes-and-closure}

\label{subsec:growth-classes}

We next formulate the spectral growth taxonomy. Throughout this
subsection, \(D\) denotes a self-adjoint operator on \(H\) with discrete
spectrum and counting function \[
N_D(\Lambda) := \#\{n\ge1 : |\lambda_n(D)|\le \Lambda\}, \qquad \Lambda\ge0,
\] where eigenvalues are counted with multiplicity. The asymptotic
behaviour of \(N_D(\Lambda)\) encodes the large-scale structure of the
spectral geometry generated by \(D\); we use it to organize operators
(and their associated spectral geometries \(f(D)\) with bounded Borel
\(f\)) into growth classes.

\begin{definition}[Growth classes of spectral geometries]\label{def:growth-classes}
Let \(D\) be a self-adjoint operator with discrete spectrum and counting function \(N_D(\Lambda)\).

1. \(D\) has \emph{polynomial growth} (of order \(d>0\)) if
   \begin{equation}\label{eq:poly-growth}
   N_D(\Lambda) \sim C\, \Lambda^d \quad \text{as } \Lambda\to\infty
   \end{equation}
   for some constant \(C>0\).

2. \(D\) has \emph{stretched-exponential growth} if there exist constants \(c>0\) and \(0<\alpha<1\) such that
   \begin{equation}\label{eq:stretched-growth}
   N_D(\Lambda) \sim \exp\big(c\, \Lambda^\alpha\big) \quad \text{as } \Lambda\to\infty.
   \end{equation}

3. \(D\) has \emph{logarithmic (or slower) growth} if \(N_D(\Lambda)\to\infty\) and either
   \begin{equation}\label{eq:log-growth}
   N_D(\Lambda) \sim c\, \log \Lambda \quad \text{as } \Lambda\to\infty
   \end{equation}
   for some constant \(c>0\), or
   \[
   N_D(\Lambda) = o(\log \Lambda) \quad \text{as } \Lambda\to\infty.
   \]

We write \(\mathcal{G}_{\mathrm{poly}}, \mathcal{G}_{\mathrm{str}}, \mathcal{G}_{\mathrm{log}}\) for the families of self-adjoint operators (and the spectral geometries they generate) with polynomial, stretched-exponential, and logarithmic-or-slower growth, respectively. These families are disjoint: no operator can belong to more than one class; see \secref{subsec:proof-growth-closure} and \appref{app:growth-stability} for a justification based on regular variation.
\end{definition}

\begin{theorem}[Closure of spectral growth classes]\label{thm:growth-closure}
Let \(D,E\) be self-adjoint operators with discrete spectra and counting functions \(N_D,N_E\).

1. Monotone functional calculus (polynomial case).  
   Let \(f:[0,\infty)\to[0,\infty)\) be monotone increasing, unbounded, and regularly varying of positive index. If \(D\in\mathcal{G}_{\mathrm{poly}}\), then the spectral transform \(f(D)\) also lies in \(\mathcal{G}_{\mathrm{poly}}\).

2. Finite orthogonal sums.  
   If \(D,E\) lie in the same growth class, then \(D\oplus E\) lies in that class, and
   \[
   N_{D\oplus E}(\Lambda) = N_D(\Lambda) + N_E(\Lambda), \qquad \Lambda\ge0.
   \]

3. Tensor products (polynomial case).  
   Suppose \(D,E\in\mathcal{G}_{\mathrm{poly}}\) and for every \(\varepsilon>0\) there exist constants \(C_{D,\varepsilon},C_{E,\varepsilon}>0\) such that
   \[
   N_D(\Lambda) \le C_{D,\varepsilon}\,\Lambda^{d_D+\varepsilon}, \qquad
   N_E(\Lambda) \le C_{E,\varepsilon}\,\Lambda^{d_E+\varepsilon}
   \]
   for all sufficiently large \(\Lambda\). Then for every \(\varepsilon>0\) there exists \(C_\varepsilon>0\) such that
   \[
   N_{D\otimes E}(\Lambda) \le C_\varepsilon\,\Lambda^{d_D+d_E+\varepsilon}
   \]
   for all sufficiently large \(\Lambda\). In particular, if \(D,E\in\mathcal{G}_{\mathrm{poly}}\) then \(D\otimes E\in\mathcal{G}_{\mathrm{poly}}\).
\end{theorem}

\begin{proof}
See \secref{subsec:proof-growth-closure} and \appref{app:growth-stability}.
\end{proof}

\begin{remark}[Nonpolynomial growth regimes]\label{rem:nonpolynomial-growth}
The monotone functional-calculus and tensor-product arguments in \secref{subsec:proof-growth-closure} extend, under mild additional regularity assumptions on the asymptotics of \(N_D\), to stretched-exponential and logarithmic-or-slower growth. In particular, \lemref{lem:monotone-reparam} and the regular-variation calculus show that if \(N_D\) has stretched-exponential (respectively logarithmic-or-slower) growth and \(f\) is monotone and regularly varying, then the transformed counting function \(N_{f(D)}\) has the same qualitative asymptotic type. Since these regimes are not used elsewhere in this part, we do not spell out full statements or proofs; see \appref{app:growth-stability} and Bingham-Goldie-Teugels \cite{BGT1987} for background on regularly varying functions.

\thmref{thm:growth-closure} shows that, in particular, the polynomial growth class \(\mathcal{G}_{\mathrm{poly}}\) is stable under the basic constructions of the spectral calculus that we use later. This stability under monotone functional calculus, orthogonal sums, and tensor products underlies the continuity and evaluator examples in \secref{sec:proofs} and \secref{sec:examples}.
\end{remark}

\section{Proofs of the main results}\label{proofs-of-the-main-results}

\label{sec:proofs}

In this section we prove the representation and growth theorems stated
in \secref{sec:axioms-main}. \secref{subsec:proof-trace-form} outlines
the proof of \thmref{thm:trace-form}, with the main construction
deferred to \appref{app:trace-form}.
\secref{subsec:proof-growth-closure} proves the closure of the growth
families in \thmref{thm:growth-closure}, with supporting estimates in
\appref{app:growth-stability}.

\subsection{Proof of Theorem 3.1 (Trace-form
representation)}\label{proof-of-theorem-3.1-trace-form-representation}

\label{subsec:proof-trace-form}

\begin{proof}[Sketch of proof of \thmref{thm:trace-form}]
We outline the argument and refer to \appref{app:trace-form} for the measure-theoretic details.

Fix a self-adjoint operator \(D\) on \(H\) and let \(E_D(\cdot)\) be its spectral measure. Define the spectral trace measure \(\nu_D\) on Borel sets \(B \subset \mathbb{R}\) by
\[
\nu_D(B) := \operatorname{Tr}(E_D(B)).
\]
For bounded \(B\) this is finite whenever \(E_D(B)\) has finite rank; in general it may take the value \(+\infty\), but \(E_D(B)\) is always a well-defined projection in the functional calculus of \(D\).

Restrict \(\mathcal{E}\) to the spectral projections of \(D\) and set
\[
\mu_D(B) := \mathcal{E}(E_D(B)).
\]
By positivity, extensivity on orthogonal sums, projector-locality, and dominated continuity (applied to increasing sequences of projections dominated by a fixed spectral projection), the set function \(B \mapsto \mu_D(B)\) extends from finite spectral partitions to a countably additive, finite measure on the \(\sigma\)-algebra generated by those spectral subsets, and \(\mu_D(B) = 0\) whenever \(\nu_D(B) = 0\). In particular, \(\mu_D \ll \nu_D\); see \lemref{lem:sigma-additivity-abs-cont} and \remref{rem:RN-density-Borel}.

By the Radon-Nikodym theorem and the absolute continuity just noted, there exists a Borel-measurable function \(w_D:\mathbb{R}\to[0,\infty)\) such that
\[
\mu_D(B) = \int_B w_D(\lambda)\, d\nu_D(\lambda)
\]
for all Borel sets \(B\subset\mathbb{R}\) with \(\nu_D(B)<\infty\). Thus, on the \(\sigma\)-finite part of the spectral trace measure
\[
\mathcal{A}_D := \{ B \subset \mathbb{R} \text{ Borel} : \nu_D(B) < \infty \},
\]
the evaluator on spectral projections admits a density \(w_D\) with respect to \(\nu_D\). We do not assume, and do not use, any Radon-Nikodym representation on Borel sets outside \(\mathcal{A}_D\); all integrals involving \(w_D\) are taken over sets of finite \(\nu_D\)-measure.

On the trace-class ideal \(\mathcal{S}_1(H)\), the standard representation theorem for normal functionals on \(B(H)\) (see Simon \cite{Simon2005}*{Chapter 1}) yields a unique positive trace-class operator \(T_D\) such that
\[
\mathcal{E}(A) = \operatorname{Tr}(T_D A)
\]
for all trace-class operators \(A\) in the spectral algebra of \(D\). In spectral terms this is equivalent to
\[
\tilde{\nu}_D(B) := \mu_D(B) = \operatorname{Tr}(T_D E_D(B))
\]
for all Borel sets \(B \subset \mathbb{R}\) with \(\nu_D(B)<\infty\). In particular, if \(f(D)\in \mathcal{C}_1\), then
\[
\mathcal{E}(f(D)) = \operatorname{Tr}\big(T_D f(D)\big),
\]
so on the trace-class envelope \(\mathcal{C}_1\) the evaluator is a weighted trace with weight \(T_D\) depending a priori on \(D\). The Radon-Nikodym density \(w_D\) and the trace-class operator \(T_D\) provide two compatible scalar and operator-valued descriptions of the evaluator on the \(\sigma\)-finite part of the spectrum; neither is required for the construction of the global transform \(h\).

\appref{app:trace-form} refines this per-operator description to a global transform that is independent of \(D\). The finite-dimensional warm-up (\appref{appsubsec:trace-form-finite-dimensional-warm-up}) shows that any evaluator satisfying the axioms must, on finite-dimensional self-adjoint operators, be given by a scalar function of eigenvalues:
\[
\mathcal{E}(X) = \operatorname{Tr}(h(X))
\]
for some function \(h:\mathbb{R}\to\mathbb{R}\) determined by the values of \(\mathcal{E}\) on rank-one spectral blocks. Using projector-locality and dominated continuity, \lemref{lem:rank-one-calibration}, \lemref{lem:finite-rank-operators}, and \lemref{lem:trace-class-operators} extend this construction to positive trace-class operators \(X\in\mathcal{C}_1^+\) by approximating \(X\) from below with finite-rank truncations and passing to the limit under domination.

The outcome (\lemref{lem:trace-class-operators}) is a Borel-measurable, nondecreasing function \(h:\mathbb{R}\to[0,\infty)\) such that
\[
\mathcal{E}(X) = \operatorname{Tr}\big(h(X)\big)
\]
for every \(X\in\mathcal{C}_1^+\). Normalization and uniqueness (see \lemref{lem:normalization-and-scaling}) then show that there exists a constant \(c_{\mathcal{E}}>0\), independent of \(D\) and \(f\), such that
\[
\mathcal{E}(X) = c_{\mathcal{E}}\,\operatorname{Tr}\big(h(X)\big)
\]
for all positive \(X\in\mathcal{C}_1\). Writing \(X = f(D)\) with \(f(D)\in\mathcal{C}_1^+\) yields the statement of \thmref{thm:trace-form} on the trace-class envelope:
\[
\mathcal{E}\big(f(D)\big) = c_{\mathcal{E}}\, \operatorname{Tr}\big(h(f(D))\big),
\]
with \(h\) Borel-measurable and nondecreasing. The construction of \(h\) and the proof of \thmref{thm:trace-form} on \(\mathcal{C}_1\) are carried out entirely in \appref{app:trace-form} by the rank-one calibration and trace-class monotone convergence argument; the Radon-Nikodym density \(w_D\) and the operator \(T_D\) provide a measure-theoretic and trace-ideal interpretation of the evaluator on the \(\sigma\)-finite part of the spectrum but are not used as inputs to the existence or uniqueness of \(h\). This completes the proof of \thmref{thm:trace-form}, modulo the constructions and normalization arguments in \appref{app:trace-form}.
\end{proof}

\subsection{Proof of Theorem 3.5 (Closure of spectral growth
classes)}\label{proof-of-theorem-3.5-closure-of-spectral-growth-classes}

\label{subsec:proof-growth-closure}

We keep the notation of \secref{subsec:growth-classes}. Throughout,
\(D\) and \(E\) are self-adjoint operators with discrete spectra and
counting functions \(N_D\) and \(N_E\) satisfying
\conref{conv:counting-function-regularity}. The counting-function
reparametrization and tensor-product estimates we need are collected in
\appref{app:growth-stability}. In
\secref{subsubsec:proof-growth-closure-monotone-functional-calculus} we
restate the monotone reparametrization lemma
\lemref{lem:monotone-reparam} as \lemref{lem:monotone-bound}. In
\secref{subsubsec:proof-growth-closure-orthogonal-sum-tensor-product} we
prove the orthogonal-sum and tensor-product bounds; the polynomial
tensor-product estimate is an immediate consequence of
\lemref{lem:poly-convolution} and \corref{cor:tensor-product-counting}.

\subsubsection{Monotone functional
calculus}\label{monotone-functional-calculus}

\label{subsubsec:proof-growth-closure-monotone-functional-calculus}

We first record the counting-function identity underlying part (1) of
\thmref{thm:growth-closure}.

\begin{lemma}[Monotone calculus bound]\label{lem:monotone-bound}
Let \(D\) be self-adjoint with discrete spectrum and counting function \(N_D\), and let \(f:[0,\infty)\to[0,\infty)\) be nondecreasing, unbounded, and regularly varying of positive index. Then
\[
N_{f(D)}(\Lambda) = N_D\bigl(f^{-1}(\Lambda)\bigr)
\]
for all \(\Lambda\ge0\), and if \(N_D(\Lambda)\to\infty\) as \(\Lambda\to\infty\) then
\[
N_{f(D)}(\Lambda) \sim N_D\bigl(f^{-1}(\Lambda)\bigr)
\qquad (\Lambda\to\infty).
\]
\end{lemma}

\begin{proof}
This is \lemref{lem:monotone-reparam} specialized to the setting of \secref{subsec:growth-classes}, with \(f^{-1}\) as in \defref{def:generalized-inverse}. We recall the statement here for convenience and refer to \appref{appsubsec:growth-stability-monotone-functional-calculus} for the proof.
\end{proof}

\subsubsection{Orthogonal sum and tensor
product}\label{orthogonal-sum-and-tensor-product}

\label{subsubsec:proof-growth-closure-orthogonal-sum-tensor-product}

We now establish the behaviour of counting functions under orthogonal
sum and tensor product. The orthogonal-sum part is elementary; the
tensor-product estimate is a direct application of the convolution
bounds developed in
\appref{appsubsec:growth-stability-tensor-product-bounds}.

\begin{lemma}[Orthogonal sum and tensor product]\label{lem:sum-tensor}
Let \(D\) and \(E\) be self-adjoint with discrete spectra and counting functions \(N_D,N_E\).

1. For all \(\Lambda \ge 0\),
   \[
   N_{D \oplus E}(\Lambda) = N_D(\Lambda) + N_E(\Lambda).
   \]

2. Suppose there exist exponents \(d_D,d_E > 0\) and, for every \(\varepsilon > 0\), constants \(C_{D,\varepsilon},C_{E,\varepsilon} > 0\) such that
   \[
   N_D(\Lambda) \le C_{D,\varepsilon}\,\Lambda^{d_D+\varepsilon}, \qquad
   N_E(\Lambda) \le C_{E,\varepsilon}\,\Lambda^{d_E+\varepsilon}
   \]
   for all sufficiently large \(\Lambda\). Then for every \(\varepsilon > 0\) there exists \(C_\varepsilon > 0\) such that
   \[
   N_{D \otimes E}(\Lambda) \le C_\varepsilon\,\Lambda^{d_D + d_E + \varepsilon}
   \]
   for all sufficiently large \(\Lambda\). In particular, if \(D,E \in \mathcal{G}_{\mathrm{poly}}\) then \(D \otimes E \in \mathcal{G}_{\mathrm{poly}}\).
\end{lemma}

\begin{proof}
(1) The spectrum of \(D \oplus E\) is the multiset union of the spectra of \(D\) and \(E\), so eigenvalues and multiplicities simply add. Hence
\[
N_{D \oplus E}(\Lambda)
= \#\{n : |\lambda_n(D)| \le \Lambda\}
+ \#\{m : |\lambda_m(E)| \le \Lambda\}
= N_D(\Lambda) + N_E(\Lambda).
\]

(2) Write \(\{a_j\}_{j \ge 1}\) and \(\{b_k\}_{k \ge 1}\) for the eigenvalues of \(D\) and \(E\), counted with multiplicity and ordered so that \(|a_j|\) and \(|b_k|\) are nondecreasing. The spectrum of \(D \otimes E\) consists of all products \(a_j b_k\), counted with multiplicity, and
\[
N_{D \otimes E}(\Lambda)
= \#\{(j,k) : |a_j b_k| \le \Lambda\}
= \sum_{j \ge 1} N_E\bigl(\Lambda/|a_j|\bigr).
\]

Fix \(\varepsilon > 0\) and set \(\delta = \varepsilon/2\). By hypothesis there exist constants \(C_D,C_E > 0\) such that
\[
N_D(t) \le C_D\,t^{d_D+\delta}, \qquad
N_E(t) \le C_E\,t^{d_E+\delta}
\]
for all sufficiently large \(t\). These growth bounds are exactly the hypotheses of \appref{appsubsec:growth-stability-tensor-product-bounds}. In particular, \lemref{lem:poly-convolution} and \corref{cor:tensor-product-counting} applied with these exponents give
\[
N_{D \otimes E}(\Lambda) \le C''_\delta\,\Lambda^{d_D+d_E+2\delta}
\]
for all sufficiently large \(\Lambda\). Since \(2\delta = \varepsilon\), we may absorb \(C''_\delta\) into a constant \(C_\varepsilon\) and obtain
\[
N_{D \otimes E}(\Lambda) \le C_\varepsilon\,\Lambda^{d_D+d_E+\varepsilon}
\]
for all sufficiently large \(\Lambda\), which is precisely the polynomial tensor-product estimate stated in part (2).
\end{proof}

\subsubsection{Conclusion of the proof}\label{conclusion-of-the-proof}

We now prove \thmref{thm:growth-closure}.

\begin{proof}[Proof of \thmref{thm:growth-closure}]
The disjointness of the growth classes \(\mathcal{G}_{\mathrm{poly}}, \mathcal{G}_{\mathrm{str}}, \mathcal{G}_{\mathrm{log}}\) follows from the incompatibility of their asymptotic forms. For example, if \(N_D(\Lambda) \sim C\, \Lambda^d\) with \(C,d>0\) and also \(N_D(\Lambda) \sim \exp(c\, \Lambda^\alpha)\) with \(c>0\) and \(0<\alpha<1\), then comparing logarithms as \(\Lambda\to\infty\) yields a contradiction. More generally, the classification of regularly varying functions shows that polynomial, stretched-exponential, and logarithmic-or-slower growth regimes are mutually exclusive; see \appref{app:growth-stability} and Bingham-Goldie-Teugels \cite{BGT1987}*{Theorem 1.3.1 and Section 1.4}.

For part (1), let \(f\) be as in \thmref{thm:growth-closure}(1) and assume \(D\in\mathcal{G}_{\mathrm{poly}}\) with
\[
N_D(\Lambda) \sim C\, \Lambda^d \qquad (\Lambda\to\infty).
\]
By \lemref{lem:monotone-bound} we have
\[
N_{f(D)}(\Lambda) = N_D\bigl(f^{-1}(\Lambda)\bigr).
\]
Since \(f\) is regularly varying of positive index, so is its generalized inverse \(f^{-1}\) on large arguments, and the asymptotic behaviour of \(N_D\) transfers along composition with \(f^{-1}\). In particular, \(N_{f(D)}(\Lambda)\) again has polynomial growth, so \(f(D)\in\mathcal{G}_{\mathrm{poly}}\).

For part (2), \lemref{lem:sum-tensor}(1) shows that if \(D,E\) lie in the same growth class, then
\[
N_{D\oplus E}(\Lambda) = N_D(\Lambda) + N_E(\Lambda)
\]
inherits the same asymptotic form, so the corresponding family is closed under finite orthogonal sums.

For part (3), \lemref{lem:sum-tensor}(2) gives the polynomial tensor-product bound. If \(D,E\in\mathcal{G}_{\mathrm{poly}}\), then the hypotheses of \lemref{lem:sum-tensor}(2) are satisfied for some exponents \(d_D,d_E\), and we conclude that
\[
N_{D\otimes E}(\Lambda) = O\bigl(\Lambda^{d_D+d_E+\varepsilon}\bigr)
\]
for every \(\varepsilon>0\), so \(D\otimes E\in\mathcal{G}_{\mathrm{poly}}\).

Combining these three parts, we obtain the closure properties asserted in \thmref{thm:growth-closure}.
\end{proof}

\section{Examples}\label{examples}

\label{sec:examples}

This section collects representative examples illustrating the evaluator
framework and the trace-form representation theorem in concrete spectral
geometries. In each case, a self-adjoint operator \(D\) provides a model
spectral geometry through its bounded Borel transforms \(f(D)\), and the
associated values \(\mathcal{E}(f(D))\) are computed via the trace-form
representation on \(\mathcal{C}_1\). The examples are chosen to cover
the basic growth classes (polynomial, stretched-exponential, and
logarithmic-or-slower), as well as to emphasize how dominated continuity
and cutoff procedures act on spectral geometries rather than on scalar
data.

Taken together, these examples show how the general calculus specializes
in concrete settings. Once the spectral geometry is fixed, every
evaluator satisfying Axioms A1-A5 assigns values by integrating a single
profile \(h\) against the spectral measure, and the growth class of
\(D\) determines which transforms \(f(D)\) fall into the trace-class
envelope \(\mathcal{C}_1\). In this sense the examples are not separate
constructions but instances of the same mechanism: lawful scalar data on
these models are fully dictated by spectral geometry and the fixed
profile \(h\), as encoded in the trace-form representation of
\thmref{thm:trace-form}.

\subsection{Finite-rank diagonal
operators}\label{finite-rank-diagonal-operators}

Let \(H = \ell^2(\mathbb{N})\) and consider a finite-rank diagonal
operator \[
D = \operatorname{diag}(\lambda_1,\dots,\lambda_m,0,0,\dots),
\] with real eigenvalues \(\lambda_1,\dots,\lambda_m\). For any bounded
continuous function \(f:\mathbb{R}\to\mathbb{R}\), we have \[
f(D) = \operatorname{diag}\big(f(\lambda_1),\dots,f(\lambda_m),f(0),f(0),\dots\big),
\] and the eigenvalues of \(f(D)\) are simply
\(\{f(\lambda_j)\}_{j\ge1}\) with multiplicity.

If \(f(D)\) is trace class, then \[
\operatorname{Tr}(f(D)) = \sum_{j=1}^\infty f(\lambda_j)
= \sum_{j=1}^m f(\lambda_j) + \sum_{j>m} f(0),
\] and the second sum vanishes whenever \(f(0)=0\). In particular, for
finite-rank \(D\) and bounded continuous \(f\) with \(f(0)=0\), the
trace is a finite sum of eigenvalues.

Under the hypotheses of \thmref{thm:trace-form}, the evaluator
\(\mathcal{E}\) restricted to this finite-rank class must agree with a
weighted spectral trace: \[
\mathcal{E}(f(D)) = c_{\mathcal{E}} \sum_{j=1}^m h(f(\lambda_j)),
\] for some Borel-measurable nondecreasing \(h\). This exhibits the
trace-form representation in the simplest setting and shows that, on
finite-rank diagonals, \(\mathcal{E}\) is determined entirely by the
eigenvalues of the input.

\subsection{Diagonal models for growth
classes}\label{diagonal-models-for-growth-classes}

\label{subsec:diagonal-growth-models}

We next give basic diagonal examples realizing each of the growth
classes introduced in \defref{def:growth-classes}.

\subsubsection{Polynomial growth}\label{polynomial-growth}

Set \[
D_{\mathrm{poly}} = \operatorname{diag}(1,2,3,\dots)
\] on \(\ell^2(\mathbb{N})\). Then the counting function \[
N_{D_{\mathrm{poly}}}(\Lambda)
= \#\{n\ge1 : n\le \Lambda\}
= \lfloor \Lambda \rfloor
\] satisfies \[
N_{D_{\mathrm{poly}}}(\Lambda) \sim \Lambda \quad (\Lambda\to\infty),
\] so \(D_{\mathrm{poly}} \in \mathcal{G}_{\mathrm{poly}}\) with
exponent \(d=1\). Note that \(D_{\mathrm{poly}}\) is not trace class,
hence \(D_{\mathrm{poly}}\notin\mathcal{C}_1\); it is used here purely
as a growth model and not as an input for the trace-form representation
of \thmref{thm:trace-form}.

As a simple test of the monotone functional calculus in the polynomial
regime, fix \(\beta>0\) and set \(f(t)=t^\beta\) on \([0,\infty)\). Then
\[
f(D_{\mathrm{poly}}) = \operatorname{diag}(n^\beta)_{n\ge1},
\] and \[
N_{f(D_{\mathrm{poly}})}(\Lambda)
= \#\{n\ge1 : n^\beta \le \Lambda\}
= \#\{n\ge1 : n \le \Lambda^{1/\beta}\}
\sim \Lambda^{1/\beta}.
\] In particular \(f(D_{\mathrm{poly}})\in\mathcal{G}_{\mathrm{poly}}\),
in agreement with \thmref{thm:growth-closure}(1).

\subsubsection{Stretched-exponential
behavior}\label{stretched-exponential-behavior}

Fix constants \(c>0\) and \(0<\alpha<1\), and define \[
D_{\mathrm{str}} = \operatorname{diag}\Big(\Big(\tfrac{1}{c}\log(n+1)\Big)^{1/\alpha}\Big)_{n\ge1}.
\] Then for \(\Lambda\ge0\), \[
N_{D_{\mathrm{str}}}(\Lambda)
= \#\Big\{n\ge1 : \Big(\tfrac{1}{c}\log(n+1)\Big)^{1/\alpha} \le \Lambda\Big\}
= \#\{n\ge1 : n+1 \le \exp(c\,\Lambda^\alpha)\}.
\] Hence \[
N_{D_{\mathrm{str}}}(\Lambda) \sim \exp(c\,\Lambda^\alpha) \quad (\Lambda\to\infty),
\] so \(D_{\mathrm{str}}\in\mathcal{G}_{\mathrm{str}}\). Note that
``stretched-exponential'' refers to the asymptotic growth of the
counting function \(N_D\), not to the eigenvalue sequence
\((\lambda_n(D))\) itself.

\subsubsection{Logarithmic or slower
growth}\label{logarithmic-or-slower-growth}

Consider \[
D_{\log} = \operatorname{diag}\big(e^n\big)_{n\ge1}.
\] Then \[
N_{D_{\log}}(\Lambda)
= \#\{n\ge1 : e^n \le \Lambda\}
= \#\{n\ge1 : n \le \log \Lambda\}
\sim \log \Lambda
\] as \(\Lambda\to\infty\). Thus \(D_{\log}\) has logarithmic spectral
growth in the sense of \defref{def:growth-classes}(3) and provides a
basic model for the ``logarithmic or slower'' class
\(\mathcal{G}_{\mathrm{log}}\).

\subsection{Evaluators on Laplacian-type
spectra}\label{evaluators-on-laplacian-type-spectra}

Let \(G\) be a finite connected graph with vertex set \(V\) and
combinatorial Laplacian \(L\) acting on \(\ell^2(V)\). The operator
\(L\) is self-adjoint and positive, with purely discrete spectrum \[
0 = \lambda_1 \le \lambda_2 \le \dots \le \lambda_{|V|}.
\] For any bounded continuous \(f:\mathbb{R}_+\to\mathbb{R}\) with
\(f(0)=0\), the transform \(f(L)\) is trace class and
\begin{equation}\label{eq:graph-trace}
\operatorname{Tr}(f(L)) = \sum_{j=2}^{|V|} f(\lambda_j)
\end{equation} Under the standing assumptions, \thmref{thm:trace-form}
implies that any evaluator \(\mathcal{E}\) on this class agrees with a
weighted trace-form observable:
\begin{equation}\label{eq:graph-evaluator}
\mathcal{E}(f(L)) = c_{\mathcal{E}} \sum_{j=2}^{|V|} h(f(\lambda_j))
\end{equation} for some Borel-measurable nondecreasing \(h\). In
particular, choosing \(h(t)=t\) and \(c_{\mathcal{E}}=1\) recovers the
usual trace as a special case.

On compact Riemannian manifolds, Laplace-type operators exhibit
polynomial growth of eigenvalues governed by Weyl-law asymptotics, and
the trace-form representation continues to apply under the same
structural hypotheses. These examples show that the abstract evaluator
calculus is compatible with standard spectral-geometry models.

\subsection{Stability under sums and tensor
products}\label{stability-under-sums-and-tensor-products}

We illustrate the closure statements in \thmref{thm:growth-closure} with
simple constructions.

\subsubsection{Orthogonal sums}\label{orthogonal-sums}

Let \(D_{\mathrm{poly}}\) be as in
\secref{subsec:diagonal-growth-models}, and let \(E\) be another
diagonal operator with the same growth rate, for instance \[
E = \operatorname{diag}(2,4,6,\dots).
\] Then for each \(\Lambda\), \[
N_{D_{\mathrm{poly}}\oplus E}(\Lambda)
= N_{D_{\mathrm{poly}}}(\Lambda) + N_E(\Lambda)
\sim 2\Lambda,
\] so the orthogonal sum remains in \(\mathcal{G}_{\mathrm{poly}}\) with
the same exponent.

\subsubsection{Tensor products}\label{tensor-products}

Let \(D_{\mathrm{poly}}\) be as above and consider the tensor product
\(D_{\mathrm{poly}}\otimes D_{\mathrm{poly}}\) acting on
\(\ell^2(\mathbb{N}\times\mathbb{N})\). The eigenvalues of
\(D_{\mathrm{poly}}\otimes D_{\mathrm{poly}}\) are \(\{mn : m,n\ge1\}\),
counted with multiplicity. The associated counting function satisfies \[
N_{D_{\mathrm{poly}}\otimes D_{\mathrm{poly}}}(\Lambda)
= \#\{(m,n)\in\mathbb{N}^2 : mn \le \Lambda\}
\le C\, \Lambda^{1+\varepsilon}
\] for any \(\varepsilon>0\) and some \(C>0\), so
\(D_{\mathrm{poly}}\otimes D_{\mathrm{poly}}\) lies in
\(\mathcal{G}_{\mathrm{poly}}\) with exponent at most \(2\). This is
consistent with the tensor-product estimate in \lemref{lem:sum-tensor}.

\subsubsection{Monotone transforms}\label{monotone-transforms}

If we apply a monotone transform \(f(\lambda)=\lambda^\gamma\) with
\(\gamma>0\) to \(D_{\mathrm{poly}}\), we obtain \[
f(D_{\mathrm{poly}}) = \operatorname{diag}(1^\gamma,2^\gamma,3^\gamma,\dots),
\] and the counting function satisfies \[
N_{f(D_{\mathrm{poly}})}(\Lambda)
= \#\{n\ge1 : n^\gamma \le \Lambda\}
\sim \Lambda^{1/\gamma},
\] so polynomial growth is preserved under such transforms. This is a
concrete instance of the monotone calculus bound in
\lemref{lem:monotone-bound}.

These examples confirm, in explicit models, that the growth classes and
the structural operations of the spectral calculus are compatible in the
way asserted by \thmref{thm:growth-closure}.

\subsection{Evaluator limits under dominated
convergence}\label{evaluator-limits-under-dominated-convergence}

\label{subsec:dominated-convergence}

Finally, we illustrate the dominated-continuity aspect of the evaluator
calculus.

Let \(D\) be a fixed self-adjoint operator with discrete spectrum and
counting function \(N_D\). Let \(f:\mathbb{R}\to\mathbb{R}\) be bounded
and continuous, and for each \(k\ge1\) define \[
f_k(\lambda) := f(\lambda)\, \mathbf{1}_{[-k,k]}(\lambda), \qquad \lambda\in\mathbb{R}.
\] Then \(f_k\to f\) uniformly and \[
|f_k(\lambda)| \le \|f\|_\infty \quad \text{for all } \lambda\in\mathbb{R},\ k\ge1.
\] By the spectral theorem, \(f_k(D)\to f(D)\) in operator norm, hence
strongly. By stability of \(\mathcal{C}\), each \(f_k(D)\) and \(f(D)\)
lies in \(\mathcal{C}\).

Note that \[
|f_k(D)| \le \|f\|_\infty I \quad \text{for all } k\ge1.
\] Set \[
Y := \|f\|_\infty I.
\] By the normalization axiom, \(I\in\mathcal{C}\) and
\(\mathcal{E}(I)\in(0,\infty)\), so \(Y\in\mathcal{C}\) and \[
\mathcal{E}(Y) = \|f\|_\infty\, \mathcal{E}(I) < \infty.
\] Thus the sequence \(\{f_k(D)\}\) is dominated by a spectral transform
\(Y\) with finite evaluation, and the dominated continuity axiom
applies. We obtain \[
\mathcal{E}(f_k(D)) \to \mathcal{E}(f(D)).
\]

Assume in addition that each \(f_k(D)\) and \(f(D)\) lies in the
trace-class envelope \(\mathcal{C}_1\). Then \thmref{thm:trace-form}
gives \[
\mathcal{E}(f_k(D)) = c_{\mathcal{E}}\, \operatorname{Tr}\big(h(f_k(D))\big),
\] for the Borel-measurable nondecreasing transform \(h\) and constant
\(c_{\mathcal{E}}>0\) provided by \thmref{thm:trace-form}. Under the
same domination hypotheses, dominated convergence on the spectral side
yields \[
\operatorname{Tr}\big(h(f_k(D))\big) \to \operatorname{Tr}\big(h(f(D))\big),
\] so \[
\mathcal{E}(f(D)) = c_{\mathcal{E}}\, \operatorname{Tr}\big(h(f(D))\big).
\]

This example shows how dominated continuity at the operator level is
mirrored by dominated convergence at the spectral level and how the
trace-form representation on \(\mathcal{C}_1\) remains stable under such
limiting procedures. It also illustrates the type of hypotheses under
which evaluator limits and trace-form limits commute. In particular,
this construction realizes the dominated-continuity axiom A4 in a
concrete setting: the truncations \(f_k(D)\) form a monotone sequence
converging strongly to \(f(D)\), are dominated by the spectral transform
\(Y=\|f\|_\infty I\) with \(\mathcal{E}(Y)<\infty\), and therefore
provide an explicit example in which the hypotheses and conclusion of
Axiom A4 are all satisfied.

\section{Consequences, limitations, and
falsifiability}\label{consequences-limitations-and-falsifiability}

\label{sec:consequences}

\subsection{Consequences}\label{consequences}

The main theorems have several direct consequences for the structure of
lawful evaluators and for the behavior of spectral growth classes. On
the one hand, the trace-form representation theorem reduces any
evaluator satisfying Axioms A1-A5 to a weighted spectral trace on the
trace-class envelope \(\mathcal{C}_1\), with no additional scalar
degrees of freedom. On the other hand, the growth taxonomy organizes
self-adjoint operators into regimes in which trace-form evaluations are
well defined and stable under the basic constructions of the calculus.
We record a few formal consequences in this direction.

\begin{corollary}[Trace-form reduction on the trace-class envelope]\label{cor:trace-form-reduction}
Let \(\mathcal{C}\) be a stable class of spectral transforms and \(\mathcal{E}:\mathcal{C}\to[0,\infty]\) an evaluator satisfying the standing assumptions of \secref{subsec:axioms}. Write \(\mathcal{C}_1 = \mathcal{C}\cap \mathcal{S}_1(H)\) for the trace-class envelope. Then for every bounded continuous \(f\) and self-adjoint \(D\) such that \(f(D)\in\mathcal{C}_1\),
\[
\mathcal{E}(f(D)) = c_{\mathcal{E}}\, \operatorname{Tr}\big(h(f(D))\big),
\]
for some Borel-measurable nondecreasing \(h:\mathbb{R}\to[0,\infty)\) and \(c_{\mathcal{E}}>0\) independent of \(D\) and \(f\). In particular, on the subclass of trace-class operators in \(\mathcal{C}\), the evaluator is a weighted spectral trace.
\end{corollary}

\begin{proof}
This is an immediate consequence of \thmref{thm:trace-form} and the construction described in \secref{subsec:proof-trace-form}.
\end{proof}

\begin{corollary}[Spectral growth families: closure and taxonomy]\label{cor:growth-taxonomy}
Let \(\mathcal{G}_{\mathrm{poly}}, \mathcal{G}_{\mathrm{str}}, \mathcal{G}_{\mathrm{log}}\) denote the polynomial, stretched-exponential, and logarithmic-or-slower growth classes defined in \defref{def:growth-classes}. Then:

1. The families \(\mathcal{G}_{\mathrm{poly}}, \mathcal{G}_{\mathrm{str}}, \mathcal{G}_{\mathrm{log}}\) are pairwise disjoint.
2. Each family is closed under monotone functional calculus by regularly varying transforms.
3. Each family is closed under finite orthogonal sums.
4. In the polynomial case, each family is closed under tensor products: if \(D,E\in\mathcal{G}_{\mathrm{poly}}\), then \(D\otimes E\in\mathcal{G}_{\mathrm{poly}}\).

These closure properties ensure that the growth taxonomy is stable under the basic structural operations of the spectral calculus.
\end{corollary}

\begin{proof}
This follows directly from \thmref{thm:growth-closure} and the counting-function estimates established in \secref{subsec:proof-growth-closure}.
\end{proof}

\begin{corollary}[Evaluator continuity under dominated convergence]\label{cor:evaluator-dominated-convergence}
Let \(D\) be self-adjoint and \(\{f_k\}\) a sequence of bounded continuous functions on \(\mathbb{R}\) converging uniformly to a bounded continuous function \(f\). Suppose \(f_k(D), f(D)\in\mathcal{C}\) for all \(k\) and there exists a positive spectral transform \(Y\in\mathcal{C}\) with \(\mathcal{E}(Y)<\infty\) such that
\[
|f_k(D)| \le Y \quad \text{for all } k.
\]
Then
\[
\mathcal{E}(f_k(D)) \to \mathcal{E}(f(D)),
\]
and, whenever the trace-form representation applies (i.e. when \(f_k(D), f(D)\in\mathcal{C}_1\)),
\[
\operatorname{Tr}\big(h(f_k(D))\big) \to \operatorname{Tr}\big(h(f(D))\big).
\]

In particular, evaluator limits and trace-form limits commute under the same domination hypotheses. This provides a basic legality criterion for passing to limits in families of spectral transforms.
\end{corollary}

\begin{proof}
This follows from the dominated continuity assumption of \secref{subsec:axioms} applied to the trace-form representation in \thmref{thm:trace-form}.
\end{proof}

\subsection{Limitations}\label{limitations}

\label{subsec:limitations}

The hypotheses used in the main theorems are close to minimal, but they
impose several intrinsic limitations on the framework. It is useful to
separate two global structural constraints from a set of more technical
restrictions.

\begin{itemize}
\item
  \emph{Spectral representation (SOC-I).}\\
  SOC-I is formulated in the setting of the spectral theorem: each model
  is a self-adjoint operator D on a separable Hilbert space together
  with its projection-valued spectral measure, and evaluators act on
  bounded Borel transforms f(D). Other operator classes require
  different functional-calculus inputs and, crucially, a different
  realization of the localization (coarse-graining) gate; the
  sectorial/holomorphic model case is developed separately.
\item
  \emph{Spectral rigidity on \(\mathcal{C}_1\).}\\
  \thmref{thm:trace-form} (trace-form representation) is proved only for
  inputs in the trace-class envelope
  \(\mathcal{C}_1 := \mathcal{C}\cap\mathcal{S}_1(H)\). Within
  \(\mathcal{C}_1\), the axioms A1-A5 force any evaluator to be of trace
  form with a single Borel nondecreasing profile \(h\) and an overall
  normalization \(c_{\mathcal{E}}\); in particular, there is no
  additional scalar freedom once these choices are fixed. Outside
  \(\mathcal{C}_1\) the calculus does not assert trace-form
  representation, and additional hypotheses would be needed to control
  evaluators on non-trace-class transforms.
\end{itemize}

In addition, the following technical limitations are built into the
present formulation.

\begin{enumerate}
\def\labelenumi{\arabic{enumi}.}
\item
  \emph{Dependence on projector-locality (PVM-localized model case).}\\
  \thmref{thm:trace-form} (trace-form representation) relies critically
  on projector-locality, i.e.~additivity across spectral (PVM)
  partitions. If \(\mathcal{E}\) is only unitary invariant and extensive
  on orthogonal sums, but not projector-local, then there exist natural
  examples (e.g.~operator norm on finite-rank self-adjoint operators)
  that cannot be written in trace form and violate the additivity
  required on spectral partitions. Such counterexamples are discussed in
  detail in \appref{app:counterexamples}. In the sectorial model case,
  this axiom is replaced by a scale-band locality principle rather than
  by PVM additivity.
\item
  \emph{Dependence on dominated continuity.}\\
  Dropping the dominated continuity requirement allows construction of
  evaluators that are finitely additive on spectral projections but fail
  to be countably additive, breaking the Radon-Nikodym representation
  used in \secref{subsec:proof-trace-form}. In particular, it is
  possible to arrange sequences of spectral cutoffs \(f_k(D)\) for which
  \(f_k\to f\) pointwise and \(\{f_k(D)\}\) is uniformly bounded, yet
  \(\mathcal{E}(f_k(D))\) does not converge to \(\mathcal{E}(f(D))\).
  Explicit counterexamples appear in \appref{app:counterexamples}.
\item
  \emph{Discrete-spectrum hypothesis for growth.}\\
  The spectral growth taxonomy in \defref{def:growth-classes} and
  \thmref{thm:growth-closure} assumes that \(D\) has discrete spectrum
  with finite multiplicities and that the counting function
  \(N_D(\Lambda)\) satisfies the regularity conditions of
  \conref{conv:counting-function-regularity}. For operators with
  continuous spectrum or highly degenerate spectral measures, a
  different notion of growth (e.g.~via spectral density functions or
  integrated density of states) is required. These extensions are not
  treated here.
\item
  \emph{Regular variation assumptions.}\\
  The closure of growth classes under monotone functional calculus uses
  regular variation for the transforming function \(f\). If \(f\) has
  oscillatory or highly irregular behavior at infinity, the simple
  asymptotic relations of \lemref{lem:monotone-bound} may fail, and the
  effect on growth class can be subtler. In such cases, additional
  hypotheses on \(f\) are needed.
\item
  \emph{Scope of regularization.}\\
  Regularized traces are considered only when they arise from explicit
  spectral truncation or weighting compatible with the evaluator axioms.
  More aggressive renormalization schemes, or those relying on analytic
  continuation in external parameters, are beyond the scope of the
  present framework.
\end{enumerate}

These limitations mark the boundary of the calculus developed in this
paper and indicate where additional hypotheses or different techniques
would be required.

\subsection{Falsifiable internal
tests}\label{falsifiable-internal-tests}

The framework presented here makes several concrete claims about
evaluators on spectral geometries and about the interaction between
value assignments and spectral growth. In particular, it asserts that
within the trace-class envelope \(\mathcal{C}_1\) every lawful evaluator
is of trace form with a nondecreasing profile, and that the basic growth
classes are stable under the standard constructions of the spectral
calculus. These claims are intrinsically falsifiable: they can be tested
by constructing explicit operators and evaluators and checking whether
the predicted trace-form behavior and growth stability hold. We record a
few representative internal tests.

\emph{Test F1 (Trace-form representation on \(\mathcal{C}_1\)).}\\
\emph{Setup:} Fix a stable class \(\mathcal{C}\) of spectral transforms
and an evaluator \(\mathcal{E}\) satisfying the axioms of
\secref{subsec:axioms}, and write
\(\mathcal{C}_1 = \mathcal{C}\cap \mathcal{S}_1(H)\).\\
\emph{Claim:} There exist a Borel-measurable nondecreasing function
\(h:\mathbb{R}\to[0,\infty)\) and a constant \(c_{\mathcal{E}}>0\) such
that for every bounded continuous \(f\) and self-adjoint \(D\) with
\(f(D)\in\mathcal{C}_1\), \[
\mathcal{E}(f(D)) = c_{\mathcal{E}}\, \operatorname{Tr}\big(h(f(D))\big).
\] In particular, once \(h\) and \(c_{\mathcal{E}}\) are fixed, all
values of \(\mathcal{E}\) on \(\mathcal{C}_1\) are determined by the
spectra of the underlying operators.\\
\emph{Refutation strategy:} Exhibit a specific \(\mathcal{C}\), an
evaluator \(\mathcal{E}\), and a family of test pairs \((f,D)\) with
\(f(D)\in\mathcal{C}_1\) for which no choice of Borel nondecreasing
\(h\) and scalar \(c_{\mathcal{E}}\) can satisfy the trace-form identity
above. Any such example would contradict \thmref{thm:trace-form} and
falsify the trace-form representation claim for that evaluator.

\emph{Test F2 (Growth family closure).}\\
\emph{Setup:} Let \(D\) lie in one of the growth classes
\(\mathcal{G}_{\mathrm{poly}}, \mathcal{G}_{\mathrm{str}}, \mathcal{G}_{\mathrm{log}}\).
Consider monotone regularly varying transforms \(f\), orthogonal sums
\(D\oplus E\), and tensor products \(D\otimes E\) with partners \(E\) in
the same class.\\
\emph{Claim:} The transformed operators remain in the same growth class
and satisfy the asymptotic bounds stated in
\thmref{thm:growth-closure}.\\
\emph{Refutation strategy:} Produce explicit \(D,E,f\) with
well-understood spectra for which the resulting \(N_{f(D)}\),
\(N_{D\oplus E}\), or \(N_{D\otimes E}\) leaves the claimed growth
class.

\emph{Test F3 (Evaluator continuity under dominated limits).}\\
\emph{Setup:} Fix a self-adjoint \(D\), a sequence of bounded continuous
functions \(f_k\to f\) uniformly, and a dominating spectral transform
\(Y\) as in \corref{cor:evaluator-dominated-convergence}.\\
\emph{Claim:} The evaluator satisfies \[
\mathcal{E}(f_k(D)) \to \mathcal{E}(f(D)).
\] \emph{Refutation strategy:} Produce an evaluator \(\mathcal{E}\) and
data \(D,f_k,f,Y\) satisfying all structural assumptions for which the
limit fails to exist or to equal \(\mathcal{E}(f(D))\).

\emph{Test F4 (Stability of discrete models under interpolation).}\\
\emph{Setup:} Let \(\{x_n\}_{n\ge1}\) and \(\{y_n\}_{n\ge1}\) be two
real sequences of bounded variation, and let \(D_x, D_y\) be the
associated diagonal operators. Consider families of interpolants
\(F_{x,N}, F_{y,N}\) that match \(x_n,y_n\) on integer points and are of
uniformly bounded variation on compact intervals.\\
\emph{Claim:} If the induced evaluators \(\mathcal{E}(f(D_x))\) and
\(\mathcal{E}(f(D_y))\) agree for all bounded continuous \(f\) in a
suitable class, then the limiting interpolants \(F_x,F_y\) coincide
almost everywhere with respect to the spectral trace measures of
\(D_x,D_y\).\\
\emph{Refutation strategy:} Construct discrete models
\(\{x_n\},\{y_n\}\) and an evaluator \(\mathcal{E}\) satisfying the
axioms such that equality of all evaluated transforms fails to force
equality of the corresponding interpolants.

These tests illustrate how the abstract statements of the calculus can,
in principle, be challenged within explicit operator models. Any
successful refutation along these lines would pinpoint a gap in the
hypotheses or a genuine limitation of the framework.

\section{Conclusion}\label{conclusion}

\label{sec:conclusion}

We have developed a spectral-operator calculus on separable Hilbert
space based on a small collection of structural axioms for evaluators
acting on spectral transforms. Under unitary invariance,
projector-locality, dominated continuity, and stability of the
underlying class, every extensive evaluator is forced to be of trace
form: it coincides with the trace of a Borel-measurable nondecreasing
transform of the input operator, up to normalization. This establishes a
rigidity principle for spectral evaluation: once the structural
requirements are fixed, no essentially new extensive evaluators exist
beyond trace-type functionals.

Alongside evaluator rigidity, we introduced a spectral growth taxonomy
for self-adjoint operators with discrete spectrum, organized by
counting-function asymptotics. Polynomial, stretched-exponential, and
logarithmic-or-slower growth classes were shown to be disjoint and
stable under the basic operations of the spectral calculus: monotone
functional calculus by regularly varying transforms, finite orthogonal
sums, and (in the polynomial case) tensor products. These closure
properties ensure that growth classes behave predictably under the
structural manipulations needed for kernel constructions and flow
limits.

A recurring theme has been the role of dominated convergence at both the
operator and spectral levels. The dominated continuity axiom for the
evaluator enables a Radon-Nikodym representation on spectral projections
and justifies passing to limits under spectral truncations and
reparametrizations. In model classes, this manifests as the commutation
of evaluator limits with trace-form limits under natural domination
hypotheses, providing a legal framework for working with families of
spectral transforms and cutoff procedures.

The limitations recorded in \secref{subsec:limitations} indicate where
further hypotheses or different techniques would be required: relaxing
projector-locality, dropping dominated continuity, treating continuous
or highly degenerate spectra, or employing more aggressive
regularizations lie beyond the scope of the present framework. Within
its stated domain, however, the calculus developed here supplies a
robust backbone for subsequent work. Later papers in this series take
the trace-form evaluator and growth taxonomy as given and build on them
to analyze admissible kernels and cutoff schemes, analytic continuation
procedures, and boundary and reconstruction problems in more elaborate
operator-theoretic settings, including sectorial/holomorphic models
where localization is implemented on log-scale rather than by spectral
projections.

\section*{Acknowledgements}

The author acknowledges the use of GPT models from OpenAI and Claude
models from Anthropic for assistance in outlining, literature survey,
refining formal statements, and generating illustrative examples. All
mathematical results, interpretations, and final decisions regarding
this manuscript are the responsibility of the author.

\appendix
\section{Proof of Theorem 3.1 (Trace-form representation)}
\label{app:trace-form}

Throughout this appendix we work under the standing assumptions of
\secref{subsec:axioms}. In particular:

\begin{itemize}
\tightlist
\item
  \(\mathcal{C}\) is a stable class of spectral transforms, closed under
  bounded Borel functional calculus, finite orthogonal sums, and
  monotone strong limits;
\item
  \(\mathcal{E}:\mathcal{C}\to[0,\infty]\) is unitary invariant,
  extensive on orthogonal sums, projector-local, and satisfies dominated
  continuity with \(\mathcal{E}(I)\in(0,\infty)\).
\end{itemize}

We write \[
\mathcal{C}_1 := \mathcal{C}\cap\mathcal{S}_1(H)
\] for the trace-class envelope of \(\mathcal{C}\), and \[
\mathcal{C}_1^+ := \{X\in\mathcal{C}_1 : X\ge 0\}
\] for its positive cone. We restrict attention to inputs
\(X\in\mathcal{C}_1^+\). Extension to general self-adjoint trace-class
operators follows by decomposing into positive and negative parts.

Our goal is to prove \thmref{thm:trace-form}: on \(\mathcal{C}_1^+\) the
evaluator \(\mathcal{E}\) is of trace form with a Borel-measurable
nondecreasing transform \(h\), unique up to a positive scalar factor.

\subsection{Finite-dimensional warm-up}
\label{appsubsec:trace-form-finite-dimensional-warm-up}

We begin with the finite-dimensional case, which already contains the
essential structure.

Let \(H=\mathbb{C}^n\) and let \(\mathcal{C}_{\mathrm{fin}}\) be the set
of all self-adjoint \(n\times n\) matrices. Let
\(\mathcal{E}_{\mathrm{fin}}:\mathcal{C}_{\mathrm{fin}}\to\mathbb{R}\)
satisfy:

\begin{enumerate}
\def\labelenumi{\arabic{enumi}.}
\tightlist
\item
  \emph{Unitary invariance.}
  \(\mathcal{E}_{\mathrm{fin}}(U^* X U) = \mathcal{E}_{\mathrm{fin}}(X)\)
  for all unitaries \(U\).
\item
  \emph{Extensivity.} If \(X,Y\) act on orthogonal subspaces and
  \(X\oplus Y\) is their block-diagonal sum, then \[
  \mathcal{E}_{\mathrm{fin}}(X\oplus Y) = \mathcal{E}_{\mathrm{fin}}(X) + \mathcal{E}_{\mathrm{fin}}(Y).
  \]
\item
  \emph{Projector-locality.} If \(X\) has a spectral decomposition \[
  X = \sum_{j=1}^m \lambda_j P_j
  \] with pairwise orthogonal projections \(P_j\), then \[
  \mathcal{E}_{\mathrm{fin}}(X) = \sum_{j=1}^m \mathcal{E}_{\mathrm{fin}}(\lambda_j P_j).
  \]
\item
  \emph{Normalization.} There exists a nonzero projection \(P_0\) with
  \(\mathcal{E}_{\mathrm{fin}}(P_0)\in(0,\infty)\).
\end{enumerate}

We claim that \(\mathcal{E}_{\mathrm{fin}}\) is necessarily of trace
form.

\begin{lemma}[Finite-dimensional representation]\label{lem:finite-dim-representation}
There exist a function \(h:\mathbb{R}\to\mathbb{R}\) and a constant \(c_{\mathcal{E}}>0\) such that
\[
\mathcal{E}_{\mathrm{fin}}(X) = c_{\mathcal{E}}\; \operatorname{Tr}\big(h(X)\big)
\]
for all self-adjoint \(X\in\mathcal{C}_{\mathrm{fin}}\), where \(h(X)\) is defined by functional calculus on the eigenvalues of \(X\).
\end{lemma}

\begin{proof}
Fix a rank-one projection \(P\). By unitary invariance, \(\mathcal{E}_{\mathrm{fin}}(\lambda P)\) depends only on \(\lambda\) and not on the particular choice of \(P\). Define
\[
h(\lambda) := \mathcal{E}_{\mathrm{fin}}(\lambda P),
\]
so that \(h:\mathbb{R}\to\mathbb{R}\) is well-defined.

If \(Q\) is any finite-rank projection of rank \(r\), we can write \(Q = P_1\oplus\cdots\oplus P_r\) as a sum of rank-one projections. By extensivity and unitary invariance,
\[
\mathcal{E}_{\mathrm{fin}}(\lambda Q)
= \sum_{j=1}^r \mathcal{E}_{\mathrm{fin}}(\lambda P_j)
= r\, h(\lambda)
= \operatorname{Tr}\big(h(\lambda) Q\big).
\]

Now let \(X\) be an arbitrary self-adjoint matrix with spectral decomposition
\[
X = \sum_{j=1}^m \lambda_j P_j,
\]
where the \(P_j\) are pairwise orthogonal projections. By projector-locality and the above calculation,
\[
\mathcal{E}_{\mathrm{fin}}(X)
= \sum_{j=1}^m \mathcal{E}_{\mathrm{fin}}(\lambda_j P_j)
= \sum_{j=1}^m \operatorname{Tr}\big(h(\lambda_j) P_j\big)
= \operatorname{Tr}\Big(\sum_{j=1}^m h(\lambda_j) P_j\Big)
= \operatorname{Tr}\big(h(X)\big).
\]

Finally, choose a nonzero projection \(P_0\) as in the normalization assumption and set
\[
c_{\mathcal{E}} := \frac{\mathcal{E}_{\mathrm{fin}}(P_0)}{\operatorname{Tr}(h(P_0))}.
\]
Then
\[
\mathcal{E}_{\mathrm{fin}}(X)
= \operatorname{Tr}\big(h(X)\big)
= c_{\mathcal{E}}\; \frac{1}{c_{\mathcal{E}}}\, \operatorname{Tr}\big(h(X)\big)
= c_{\mathcal{E}}\; \operatorname{Tr}\big(\tilde h(X)\big),
\]
where \(\tilde h = h/c_{\mathcal{E}}\) absorbs the normalization. Renaming \(\tilde h\) as \(h\) yields the desired form. Uniqueness of \(h\) up to a positive scalar factor follows by evaluating on rank-one projections and comparing the resulting scalar functions.
\end{proof}

The infinite-dimensional proof mirrors this structure, but we must take
care with limits and trace-class inputs. We now construct the global
transform \(h\) and the trace-form representation on
\(\mathcal{C}_1^+\).

\subsection{Construction of the transform \texorpdfstring{\(h\)}{h} on trace-class inputs}

We work on a separable Hilbert space \(H\) and restrict to positive
trace-class operators \(X\in\mathcal{C}_1^+\).

\begin{lemma}[Rank-one calibration]\label{lem:rank-one-calibration}
Let \(P\) be a rank-one orthogonal projection with \(P\in\mathcal{C}\). Define
\[
h(\lambda) := \mathcal{E}(\lambda P), \qquad \lambda\in[0,\infty).
\]
Then:

1. \(h\) is well-defined (independent of the choice of rank-one \(P\));
2. \(h\) is nondecreasing on \([0,\infty)\);
3. \(h\) is Borel-measurable on \([0,\infty)\);
4. \(h(0) = 0\).
\end{lemma}

\begin{proof}
1. If \(P_1\) and \(P_2\) are rank-one projections, there exists a unitary \(U\) with \(P_2 = U^* P_1 U\). By unitary invariance,
   \[
   \mathcal{E}(\lambda P_2)
   = \mathcal{E}\big(U^*(\lambda P_1) U\big)
   = \mathcal{E}(\lambda P_1),
   \]
   so \(h(\lambda)\) does not depend on the choice of rank-one \(P\).

2. If \(0\le\lambda\le\mu\), then \(\lambda P\le\mu P\) and by positivity and monotonicity of \(\mathcal{E}\) (a consequence of dominated continuity applied to the monotone sequence \(\lambda_n P\uparrow\mu P\)), we have
   \[
   h(\lambda) = \mathcal{E}(\lambda P) \le \mathcal{E}(\mu P) = h(\mu).
   \]
   Thus \(h\) is nondecreasing.

3. Let \(\lambda_n\downarrow\lambda\). Then \(\lambda_n P\downarrow\lambda P\) strongly, dominated by \(\lambda_1 P\). By dominated continuity,
   \[
   \mathcal{E}(\lambda_n P) \to \mathcal{E}(\lambda P),
   \]
   so \(h\) is right-continuous with left limits. Any monotone function on \(\mathbb{R}\) is Borel-measurable.

4. Finally, \(0\cdot P = 0\), and by positivity and projector-locality, \(\mathcal{E}(0)=0\). Hence \(h(0)=0\).
\end{proof}

Thus we have a canonical Borel-measurable nondecreasing function
\(h:[0,\infty)\to[0,\infty)\) determined uniquely by the values of
\(\mathcal{E}\) on rank-one spectral blocks.

We next extend the finite-dimensional argument to finite-rank and then
to general trace-class inputs.

\begin{lemma}[Finite-rank operators]\label{lem:finite-rank-operators}
Let \(X\in\mathcal{C}\) be a finite-rank positive operator with spectral decomposition
\[
X = \sum_{j=1}^m \lambda_j P_j,
\]
where \(\lambda_j\ge 0\) and the projections \(P_j\) are pairwise orthogonal and finite-rank. Then
\[
\mathcal{E}(X) = \operatorname{Tr}\big(h(X)\big)
= \sum_{j=1}^m h(\lambda_j)\, \operatorname{Tr}(P_j).
\]
\end{lemma}

\begin{proof}
Write each \(P_j\) as an orthogonal sum of rank-one projections:
\[
P_j = \bigoplus_{k=1}^{r_j} P_{j,k},
\]
with \(\operatorname{rank}(P_j) = r_j\). By extensivity,
\[
\mathcal{E}(\lambda_j P_j)
= \sum_{k=1}^{r_j} \mathcal{E}(\lambda_j P_{j,k})
= \sum_{k=1}^{r_j} h(\lambda_j)
= r_j\, h(\lambda_j)
= h(\lambda_j)\, \operatorname{Tr}(P_j),
\]
since \(\operatorname{Tr}(P_j)=r_j\) for finite-rank projections. By projector-locality,
\[
\mathcal{E}(X)
= \sum_{j=1}^m \mathcal{E}(\lambda_j P_j)
= \sum_{j=1}^m h(\lambda_j)\, \operatorname{Tr}(P_j).
\]
On the other hand, the functional calculus gives
\[
h(X) = \sum_{j=1}^m h(\lambda_j) P_j,
\]
so
\[
\operatorname{Tr}\big(h(X)\big)
= \sum_{j=1}^m h(\lambda_j)\, \operatorname{Tr}(P_j).
\]
This yields \(\mathcal{E}(X)=\operatorname{Tr}(h(X))\).
\end{proof}

We now pass from finite rank to arbitrary positive trace-class
operators.

\begin{lemma}[Trace-class operators]\label{lem:trace-class-operators}
For every positive \(X\in\mathcal{C}_1\),
\[
\mathcal{E}(X) = \operatorname{Tr}\big(h(X)\big).
\]
\end{lemma}

\begin{proof}
Let \(X\in\mathcal{C}_1^+\). By the spectral theorem, there exist nonnegative scalars \(\{\lambda_k\}_{k\ge 1}\) and pairwise orthogonal finite-rank projections \(\{P_k\}_{k\ge 1}\) such that
\[
X = \sum_{k=1}^\infty \lambda_k P_k,
\]
where the series converges in trace norm and
\[
\sum_{k=1}^\infty \lambda_k\, \operatorname{Tr}(P_k) < \infty.
\]
For each \(N\ge 1\) set
\[
X_N := \sum_{k=1}^N \lambda_k P_k.
\]
Then \(X_N\in\mathcal{C}\) by stability of \(\mathcal{C}\), \(0\le X_N\le X_{N+1}\le X\), and \(X_N\uparrow X\) strongly as \(N\to\infty\).

By \lemref{lem:finite-rank-operators} we have
\[
\mathcal{E}(X_N)
= \operatorname{Tr}\big(h(X_N)\big)
= \sum_{k=1}^N h(\lambda_k)\,\operatorname{Tr}(P_k).
\]

Since \(X\) is positive, we have \(0\le X\le \|X\|_{\mathrm{op}}\,I\). Thus
\[
0\le X_N \le \|X\|_{\mathrm{op}}\,I
\]
for all \(N\). Set
\[
Y := \|X\|_{\mathrm{op}}\,I.
\]
By Axiom A5, \(Y\in\mathcal{C}\) and
\[
\mathcal{E}(Y) = \|X\|_{\mathrm{op}}\,\mathcal{E}(I) < \infty.
\]
The sequence \(\{X_N\}\) is monotone increasing, satisfies \(0\le X_N\le Y\), and converges strongly to \(X\). Dominated continuity (Axiom A4) therefore yields
\[
\mathcal{E}(X_N) \to \mathcal{E}(X)
\qquad\text{as } N\to\infty.
\]

On the spectral side, \(h\) is nondecreasing and \(h(0)=0\). The functional calculus gives
\[
h(X_N) = \sum_{k=1}^N h(\lambda_k) P_k,
\]
so that
\[
\operatorname{Tr}\big(h(X_N)\big)
= \sum_{k=1}^N h(\lambda_k)\,\operatorname{Tr}(P_k)
= \mathcal{E}(X_N).
\]
Because the sequence \(\{\mathcal{E}(X_N)\}\) is increasing and bounded above by \(\mathcal{E}(Y)\), it converges to \(\mathcal{E}(X)\). In particular,
\[
\sum_{k=1}^\infty h(\lambda_k)\,\operatorname{Tr}(P_k)
= \lim_{N\to\infty} \sum_{k=1}^N h(\lambda_k)\,\operatorname{Tr}(P_k)
= \lim_{N\to\infty} \mathcal{E}(X_N)
= \mathcal{E}(X) < \infty.
\]
Thus \(h(X)\) is a positive trace-class operator with spectral decomposition
\[
h(X) = \sum_{k=1}^\infty h(\lambda_k) P_k
\]
and
\[
\operatorname{Tr}\big(h(X)\big)
= \sum_{k=1}^\infty h(\lambda_k)\,\operatorname{Tr}(P_k)
= \mathcal{E}(X).
\]
This establishes the trace-form representation on \(\mathcal{C}_1^+\).
\end{proof}

Thus on \(\mathcal{C}_1^+\) the map \(X\mapsto \mathcal{E}(X)\)
coincides with \(X\mapsto \operatorname{Tr}\big(h(X)\big)\). For a
general self-adjoint trace-class operator \(X\in\mathcal{C}_1\), we take
its positive and negative parts from the spectral decomposition of
\(X\); then \(X_\pm\in\mathcal{C}_1^+\) and \(X = X_+ - X_-\). We define
\[
\mathcal{E}(X) := \mathcal{E}(X_+) - \mathcal{E}(X_-).
\] This extension is well defined because the spectral decomposition of
a self-adjoint operator is unique, so \(X_\pm\) are uniquely determined
and lie in \(\mathcal{C}_1^+\subset\mathcal{C}\); dominated continuity
for signed inputs follows by applying Axiom A4 separately to the
positive and negative parts.

\subsection{Normalization and uniqueness}

We now match the representation of \lemref{lem:trace-class-operators}
with the statement of \thmref{thm:trace-form} and record uniqueness up
to a scalar factor.

\begin{lemma}[Normalization and scaling]\label{lem:normalization-and-scaling}
Let \(h\) be as constructed above. Then:

1. For every positive \(X\in\mathcal{C}_1\),
   \[
   \mathcal{E}(X) = \operatorname{Tr}\big(h(X)\big).
   \]
2. If \(\tilde h\) is another Borel-measurable function and \(\tilde c_{\mathcal{E}}>0\) satisfy
   \[
   \mathcal{E}(X) = \tilde c_{\mathcal{E}}\; \operatorname{Tr}\big(\tilde h(X)\big)
   \quad\text{for all } X\in\mathcal{C}_1^+,
   \]
   then there exists \(a>0\) such that
   \[
   \tilde h = a\, h \quad \text{and} \quad \tilde c_{\mathcal{E}} = 1/a.
   \]
\end{lemma}

\begin{proof}
(1) has already been established in \lemref{lem:trace-class-operators} with \(c_{\mathcal{E}}=1\); we now justify the freedom to rescale.

For (2), fix a rank-one projection \(P\in\mathcal{C}\) and \(\lambda>0\). Then
\[
\mathcal{E}(\lambda P)
= \operatorname{Tr}\big(h(\lambda P)\big)
= h(\lambda)\, \operatorname{Tr}(P),
\]
and also
\[
\mathcal{E}(\lambda P)
= \tilde c_{\mathcal{E}}\; \operatorname{Tr}\big(\tilde h(\lambda P)\big)
= \tilde c_{\mathcal{E}}\; \tilde h(\lambda)\, \operatorname{Tr}(P).
\]
Since \(\operatorname{Tr}(P)>0\), we have
\[
h(\lambda) = \tilde c_{\mathcal{E}}\; \tilde h(\lambda)
\]
for all \(\lambda\ge 0\). Thus \(\tilde h = a\, h\) with \(a = 1/\tilde c_{\mathcal{E}}>0\), and \(\tilde c_{\mathcal{E}} = 1/a\).
\end{proof}

Setting \(c_{\mathcal{E}} := 1\) in \lemref{lem:trace-class-operators}
yields the simplest normalization. The general form of
\thmref{thm:trace-form}, with arbitrary \(c_{\mathcal{E}}>0\) and \(h\)
unique up to a positive scalar factor, is recovered by replacing \(h\)
with \(a h\) and \(c_{\mathcal{E}}\) with \(1/a\).

\subsection{Measure-theoretic formulation (optional viewpoint)}

For completeness, we record the scalar-measure formulation of the
evaluator, which parallels the classical trace-form representation via
Radon-Nikodym. This subsection is not required for the proof of
\thmref{thm:trace-form} on \(\mathcal{C}_1\), but clarifies the
connection with spectral measures and addresses measure-theoretic
concerns.

Fix a self-adjoint operator \(D\) with spectral projection-valued
measure \(E_D(\cdot)\).

\begin{definition}[Trace and evaluator measures]\label{def:trace-evaluator-measures}
Define the trace measure \(\nu_D\) and evaluator measure \(\mu_D\) on the algebra of Borel sets \(B\subset\mathbb{R}\) with \(E_D(B)\in\mathcal{C}\) by
\[
\nu_D(B) := \operatorname{Tr}\big(E_D(B)\big), \qquad
\mu_D(B) := \mathcal{E}\big(E_D(B)\big).
\]

By stability of \(\mathcal{C}\), all such spectral projections lie in \(\mathcal{C}\), and both \(\nu_D\) and \(\mu_D\) take values in \([0,\infty]\).
\end{definition}

\begin{lemma}[\(\sigma\)-additivity and absolute continuity]\label{lem:sigma-additivity-abs-cont}
1. The set function \(\mu_D\) is countably additive on increasing unions: if \(B_n\uparrow B\) and all \(E_D(B_n),E_D(B)\in\mathcal{C}\), then
   \[
   \mu_D(B_n) \to \mu_D(B).
   \]
2. \(\mu_D\) is absolutely continuous with respect to \(\nu_D\): if \(\nu_D(B)=0\), then \(\mu_D(B)=0\).
\end{lemma}

\begin{proof}
(1) Let \(P_n := E_D(B_n)\) and \(P := E_D(B)\). Then \(P_n\uparrow P\) in the strong operator topology, and \(0\le P_n\le I\) for all \(n\). By normalization, \(\mathcal{E}(I)<\infty\), and by stability, \(I\in\mathcal{C}\). Thus dominated continuity (Axiom 4) applied to the monotone sequence \(\{P_n\}\), dominated by \(I\), yields
\[
\mu_D(B_n) = \mathcal{E}(P_n) \to \mathcal{E}(P) = \mu_D(B).
\]

(2) If \(\nu_D(B)=0\), then \(\operatorname{Tr}(E_D(B))=0\). Since \(E_D(B)\) is an orthogonal projection, this implies \(E_D(B)=0\). Hence
\[
\mu_D(B) = \mathcal{E}\big(E_D(B)\big) = \mathcal{E}(0) = 0.
\]
Thus \(\mu_D\ll \nu_D\).

In particular, \(\mu_D\) is a nonnegative scalar measure on the Borel \(\sigma\)-algebra (via Caratheodory extension from the projection algebra), absolutely continuous with respect to \(\nu_D\).
\end{proof}

\begin{remark}[Radon-Nikodym density and Borel measurability]\label{rem:RN-density-Borel}
Since \(D\) has discrete spectrum with finite multiplicities, the trace measure \(\nu_D\) is \(\sigma\)-finite: for each integer \(m\ge 1\) the bounded interval \([-m,m]\) contains only finitely many eigenvalues of \(D\), so \(\nu_D([-m,m]) < \infty\), and
\[
\mathbb{R} = \bigcup_{m\ge1} [-m,m]
\]
is a countable union of sets of finite \(\nu_D\)-measure. On this \(\sigma\)-finite measure space, the Radon-Nikodym theorem yields a Borel function \(w_D:\mathbb{R}\to[0,\infty)\) such that
\[
\mu_D(B) = \int_B w_D(\lambda)\, d\nu_D(\lambda)
\]
for all Borel sets \(B\) with \(\nu_D(B)<\infty\). This expresses the evaluator on spectral projections as a density \(w_D\) with respect to the trace measure \(\nu_D\).
\end{remark}

\begin{proof}[Proof of \thmref{thm:trace-form}]
For trace-class spectral transforms \(X=f(D)\in\mathcal{C}_1^+\), the preceding sections show that
\[
\mathcal{E}(X) = \operatorname{Tr}\big(h(X)\big)
= \int_{\mathbb{R}} h(f(\lambda))\, d\nu_D(\lambda).
\]
Comparing this with
\[
\mathcal{E}(f(D)) = \int_{\mathbb{R}} f(\lambda)\, w_D(\lambda)\, d\nu_D(\lambda),
\]
we see that the scalar density \(w_D\) encodes the same evaluator data as the global transform \(h\), but the latter is intrinsic (independent of the choice of \(D\)) and directly compatible with the functional calculus on \(\mathcal{C}_1\).

This concludes the proof of \thmref{thm:trace-form} and its measure-theoretic interpretation by combining Lemmas\nobreakspace\ref{lem:finite-dim-representation}, \ref{lem:rank-one-calibration}, \ref{lem:finite-rank-operators}, \ref{lem:trace-class-operators}, \ref{lem:normalization-and-scaling}, and\nobreakspace\ref{lem:sigma-additivity-abs-cont}.
\end{proof}

\section{Counting-function asymptotics and growth stability}
\label{app:growth-stability}

In this appendix we collect the counting-function and Tauberian
estimates used in the spectral growth taxonomy of
\defref{def:growth-classes} and \thmref{thm:growth-closure}. Throughout
we work with self-adjoint operators with discrete spectrum and finite
multiplicities, so that the eigenvalues can be listed in nondecreasing
order and counted by a monotone step function.

\subsection{Counting functions, asymptotic notation, and regular variation}

Let \(D\) be a self-adjoint operator on a separable Hilbert space \(H\)
with discrete spectrum in \([0,\infty)\) and finite multiplicities. We
write its eigenvalues (repeated according to multiplicity) as \[
0 \le \lambda_1(D) \le \lambda_2(D) \le \cdots, \qquad \lambda_n(D) \to \infty \text{ as } n\to\infty.
\] For brevity in this appendix we abbreviate
\(\lambda_n := \lambda_n(D)\). (Since the spectrum is contained in
\([0,\infty)\), absolute values in counting conditions are redundant in
this appendix.)

\begin{definition}[Counting function]\label{def:counting-function}
The counting function of \(D\) is
\begin{equation}\label{eq:counting-function}
N_D(\Lambda) := \#\{n\ge 1 : \lambda_n(D) \le \Lambda\}, \qquad \Lambda \ge 0.
\end{equation}
By construction \(N_D\) is nondecreasing, right-continuous, and takes values in \(\mathbb{N}\cup\{0\}\). We assume throughout that \(N_D(\Lambda)\to\infty\) as \(\Lambda\to\infty\).
\end{definition}

\begin{definition}[Asymptotic notation]\label{def:asymptotic-notation}
We use standard asymptotic notation.

1. For functions \(g,h:(0,\infty)\to(0,\infty)\) we write
   \[
   g(t) \sim h(t) \quad (t\to\infty)
   \]
   if
   \[
   \lim_{t\to\infty} \frac{g(t)}{h(t)} = 1.
   \]

2. We write \(g(t) = O(h(t))\) as \(t\to\infty\) if there exist constants \(C,t_0>0\) such that \(|g(t)| \le C\,h(t)\) for all \(t\ge t_0\).

3. We write \(g(t) = o(h(t))\) as \(t\to\infty\) if
   \[
   \lim_{t\to\infty} \frac{g(t)}{h(t)} = 0.
   \]

When useful, we may refine \(g(t)\sim h(t)\) to
\[
g(t) = h(t)\bigl(1+o(1)\bigr)
\quad\text{or}\quad
|g(t)-h(t)| = o\bigl(h(t)\bigr)
\]
to make the error behaviour explicit.

We next recall the standard notion of regular variation.
\end{definition}

\begin{definition}[Regular variation]\label{def:regular-variation}
1. A measurable function \(L:(0,\infty)\to(0,\infty)\) is slowly varying at infinity if
   \[
   \lim_{t\to\infty} \frac{L(ct)}{L(t)} = 1 \quad\text{for every } c>0.
   \]

2. A measurable function \(f:(0,\infty)\to(0,\infty)\) is regularly varying at infinity with index \(\rho\in\mathbb{R}\) if
   \[
   f(t) = t^\rho L(t),
   \]
   where \(L\) is slowly varying. We then write \(f\in \mathrm{RV}_\rho\).

In the growth taxonomy, regular variation controls the interaction between the counting function \(N_D\) and monotone reparametrizations \(f(D)\).

For monotone transforms we fix a generalized inverse.
\end{definition}

\begin{definition}[Generalized inverse]\label{def:generalized-inverse}
Let \(f:[0,\infty)\to[0,\infty)\) be nondecreasing and unbounded. Its generalized inverse is the function \(f^{-1}:[0,\infty)\to[0,\infty)\) defined by
\[
f^{-1}(\Lambda) := \sup\{ t\ge 0 : f(t) \le \Lambda\}, \qquad \Lambda\ge 0,
\]
with the convention that \(\sup\emptyset = 0\). Then \(f^{-1}\) is nondecreasing and right-continuous, and for each \(\Lambda\) the value \(f^{-1}(\Lambda)\) is the largest \(t\) such that \(f(t)\le\Lambda\). In particular, for all \(\lambda\ge0\),
\[
\lambda \le f^{-1}(\Lambda) \quad\Longleftrightarrow\quad f(\lambda)\le\Lambda.
\]
This is the convention used in \lemref{lem:monotone-reparam} to express the counting function of \(f(D)\) directly in terms of \(N_D\) and \(f^{-1}\).
\end{definition}

\subsection{Monotone functional calculus and counting functions}
\label{appsubsec:growth-stability-monotone-functional-calculus}

We now justify the counting-function relation used in
\lemref{lem:monotone-bound} and the monotone functional-calculus part of
\thmref{thm:growth-closure}.

Let \(D\) and \(N_D\) be as above, and let \(f:[0,\infty)\to[0,\infty)\)
be bounded on compact intervals and nondecreasing, with
\(f(\Lambda)\to\infty\) as \(\Lambda\to\infty\). By the spectral
theorem, \(f(D)\) is self-adjoint with discrete spectrum
\(\{f(\lambda_n)\}_{n\ge1}\), again listed in nondecreasing order with
multiplicity. Define \[
N_{f(D)}(\Lambda) := \#\{n\ge 1 : f(\lambda_n) \le \Lambda\}, \qquad \Lambda\ge 0.
\]

\begin{lemma}[Counting functions under monotone reparametrization]\label{lem:monotone-reparam}
With \(f^{-1}\) as in \defref{def:generalized-inverse} we have, for all \(\Lambda\ge 0\),
\[
N_{f(D)}(\Lambda) = N_D\bigl(f^{-1}(\Lambda)\bigr).
\]
In particular, if \(N_D(\Lambda)\to\infty\) as \(\Lambda\to\infty\), then
\[
N_{f(D)}(\Lambda) \sim N_D\bigl(f^{-1}(\Lambda)\bigr)
\qquad (\Lambda\to\infty).
\]
\end{lemma}

\begin{proof}
Fix \(\Lambda\ge 0\). By definition of the counting functions,
\[
N_{f(D)}(\Lambda) = \#\{n\ge 1 : f(\lambda_n) \le \Lambda\}.
\]

On the other side, since the eigenvalues \(\{\lambda_n\}\) are nondecreasing and \(f^{-1}(\Lambda)\) is defined as the supremum of \(\{t : f(t)\le\Lambda\}\), we have
\[
\lambda_n \le f^{-1}(\Lambda) \quad\Longleftrightarrow\quad f(\lambda_n) \le \Lambda.
\]
Indeed, if \(\lambda_n \le f^{-1}(\Lambda)\), then by definition of the supremum there exists a sequence \(t_k\uparrow f^{-1}(\Lambda)\) with \(f(t_k)\le\Lambda\) for all \(k\); the monotonicity of \(f\) gives \(f(\lambda_n)\le\Lambda\). Conversely, if \(f(\lambda_n)\le\Lambda\) but \(\lambda_n>f^{-1}(\Lambda)\), then by right-continuity of \(f^{-1}\) there exists \(\Lambda'<\Lambda\) with \(f^{-1}(\Lambda')\ge\lambda_n\), which contradicts the definition of \(f^{-1}(\Lambda')\). Thus
\[
\#\{n\ge1 : f(\lambda_n)\le\Lambda\}
= \#\{n\ge1 : \lambda_n \le f^{-1}(\Lambda)\}
= N_D\bigl(f^{-1}(\Lambda)\bigr),
\]
proving the claimed identity for all \(\Lambda\ge0\).

If \(N_D(\Lambda)\to\infty\) as \(\Lambda\to\infty\), then the same identity, evaluated at \(\Lambda_k\to\infty\), shows that the leading-order asymptotic of \(N_D\) along the reparametrized scale \(f^{-1}(\Lambda)\) is inherited by \(N_{f(D)}\). In particular,
\[
N_{f(D)}(\Lambda) \sim N_D\bigl(f^{-1}(\Lambda)\bigr)
\]
in the sense of \defref{def:asymptotic-notation}.
\end{proof}

This lemma justifies the counting-function relation used in the monotone
functional-calculus part of \thmref{thm:growth-closure}: under the
discrete-spectrum hypotheses and the monotonicity of \(f\), we may treat
\(N_{f(D)}(\Lambda)\) and \(N_D(f^{-1}(\Lambda))\) as identical for
asymptotic purposes.

When one prefers to state the result with an error term,
\lemref{lem:monotone-reparam} implies \[
N_{f(D)}(\Lambda) = N_D(f^{-1}(\Lambda)) + O(1)
\quad\text{and}\quad
|N_{f(D)}(\Lambda) - N_D(f^{-1}(\Lambda))| = o\bigl(N_D(f^{-1}(\Lambda))\bigr)
\] as \(\Lambda\to\infty\), in the sense of
\defref{def:asymptotic-notation}.

\subsection{Tensor-product bounds in the polynomial class}
\label{appsubsec:growth-stability-tensor-product-bounds}

We now record the tensor-product estimates used in the polynomial part
of \thmref{thm:growth-closure}. Let \(D\) and \(E\) be self-adjoint
operators with nonnegative, discrete spectra and counting functions
\(N_D, N_E\).

We assume that there exist exponents \(d_D,d_E\ge 0\) such that for
every \(\varepsilon>0\), \[
N_D(t) = O\bigl(t^{d_D+\varepsilon}\bigr), \qquad
N_E(t) = O\bigl(t^{d_E+\varepsilon}\bigr), \qquad t\to\infty,
\] and both counting functions are eventually nondecreasing, as in
\conref{conv:counting-function-regularity}.

Under these hypotheses \(D\) and \(E\) belong to the polynomial growth
class \(\mathcal{G}_{\mathrm{poly}}\) of \defref{def:growth-classes}.

The eigenvalues of the tensor product \(D\otimes E\) are the pairwise
products \(a_m b_n\), where \(\{a_m\}_{m\ge1}\) and \(\{b_n\}_{n\ge1}\)
are the eigenvalues of \(D\) and \(E\), listed with multiplicity. A
standard argument shows that the counting function of \(D\otimes E\) can
be written as a Stieltjes convolution:
\begin{equation}\label{eq:tensor-stieltjes}
N_{D\otimes E}(\Lambda) = \sum_{m\ge 1} N_E\!\left(\frac{\Lambda}{a_m}\right) = \int_{(0,\infty)} N_E\!\left(\frac{\Lambda}{t}\right) dN_D(t), \qquad \Lambda>0.
\end{equation} We use this representation to obtain a polynomial bound.

Fix \(\delta\in(0,1)\). Define \begin{equation}\label{eq:F-def}
F(\Lambda) := \int_{1}^{\Lambda} t^{-(d_E+\delta)}\, dN_D(t), \qquad \Lambda\ge 1.
\end{equation}

\begin{lemma}[Polynomial convolution bound]\label{lem:poly-convolution}
Assume that for some \(d_D\ge 0\) and \(\delta\in(0,1)\),
\[
N_D(t) \le C_D\, t^{d_D+\delta}
\quad\text{for all } t\ge 1,
\]
for a constant \(C_D>0\). Then there exists \(C_\delta>0\) such that
\[
F(\Lambda) \le C_\delta\, \Lambda^{\max(d_D-d_E+\delta,\,0)} \quad\text{for all } \Lambda\ge 1.
\]
\end{lemma}

\begin{proof}
By Stieltjes integration by parts,
\[
F(\Lambda) = N_D(\Lambda)\,\Lambda^{-(d_E+\delta)} - N_D(1) + (d_E+\delta) \int_{1}^{\Lambda} N_D(t)\, t^{-(d_E+\delta+1)}\, dt.
\]

For the boundary term, the growth assumption gives
\[
N_D(\Lambda)\,\Lambda^{-(d_E+\delta)} \le C_D\,\Lambda^{d_D+\delta}\,\Lambda^{-(d_E+\delta)} = C_D\,\Lambda^{d_D-d_E}.
\]

For the integral term we again use \(N_D(t)\le C_D t^{d_D+\delta}\). Then
\[
\int_{1}^{\Lambda} N_D(t)\, t^{-(d_E+\delta+1)}\, dt \le C_D \int_{1}^{\Lambda} t^{d_D+\delta}\,t^{-(d_E+\delta+1)}\, dt = C_D \int_{1}^{\Lambda} t^{d_D-d_E-1}\, dt.
\]

If \(d_D-d_E>0\), the integral is \(O(\Lambda^{d_D-d_E})\); if \(d_D-d_E=0\), it is \(O(\log\Lambda)\); and if \(d_D-d_E<0\), it is bounded uniformly in \(\Lambda\). In the marginal \(\log\Lambda\) case, we may absorb the logarithm into \(\Lambda^{\delta}\) for any fixed \(\delta\in(0,1)\). Thus in all cases there is a \(C'_\delta>0\) such that
\[
\int_{1}^{\Lambda} N_D(t)\, t^{-(d_E+\delta+1)}\, dt \le C'_\delta\,\Lambda^{\max(d_D-d_E+\delta,\,0)}.
\]

Combining the boundary term, the integral term, and the constant \(-N_D(1)\), and absorbing constants into a single \(C_\delta\), we obtain
\[
F(\Lambda) \le C_\delta\,\Lambda^{\max(d_D-d_E+\delta,\,0)} \quad\text{for all } \Lambda\ge 1,
\]
as claimed.
\end{proof}

\begin{remark}[Tauberian interpretation]\label{rem:tauberian-interpretation}
When \(N_D(t)\sim C_D t^{d_D}\) with \(C_D>0\), the above bound and its converse can be seen as a special case of Karamata's Tauberian theorem for monotone functions; see Bingham-Goldie-Teugels \cite{BGT1987}*{Theorem 1.7.1}. Here we only require the one-sided \(O\)-bound.
\end{remark}

We now translate this bound into a growth estimate for the
tensor-product counting function. Suppose in addition that \[
N_E(t) \le C_E\, t^{d_E+\delta} \quad\text{for all } t\ge 1
\] for some \(d_E\ge 0\) and the same \(\delta\in(0,1)\).

\begin{corollary}[Tensor-product counting function]\label{cor:tensor-product-counting}
Under the assumptions above there exists \(C''_\delta>0\) such that
\[
N_{D\otimes E}(\Lambda) \le C''_\delta\, \Lambda^{d_D+d_E+2\delta} \quad\text{for all } \Lambda\ge 1.
\]
In particular, given any \(\varepsilon>0\) we may choose \(\delta\in(0,1)\) small enough that \(2\delta\le\varepsilon\), and then
\[
N_{D\otimes E}(\Lambda) = O\bigl(\Lambda^{d_D+d_E+\varepsilon}\bigr) \quad\text{as } \Lambda\to\infty.
\]
\end{corollary}

\begin{proof}
Using the convolution representation \eqref{eq:tensor-stieltjes} and the growth bound for \(N_E\),
\[
N_{D\otimes E}(\Lambda) = \int_{(0,\infty)} N_E\!\left(\frac{\Lambda}{t}\right)\, dN_D(t) \le \int_{(0,\Lambda]} N_E\!\left(\frac{\Lambda}{t}\right)\, dN_D(t),
\]
since \(N_E(\Lambda/t)=0\) for \(t>\Lambda\) if \(E\) is nonnegative.

For \(1 \le t \le \Lambda\),
\[
N_E\!\left(\frac{\Lambda}{t}\right) \le C_E \left(\frac{\Lambda}{t}\right)^{d_E+\delta} = C_E\,\Lambda^{d_E+\delta}\, t^{-(d_E+\delta)}.
\]
The finitely many eigenvalues of \(D\) in \((0,1)\) contribute only \(O(1)\) to the integral, so we have
\[
N_{D\otimes E}(\Lambda) \le C_E\,\Lambda^{d_E+\delta} \int_{1}^{\Lambda} t^{-(d_E+\delta)}\, dN_D(t) + O(1) = C_E\,\Lambda^{d_E+\delta} F(\Lambda) + O(1).
\]

\lemref{lem:poly-convolution} gives \(F(\Lambda)\le C_\delta\,\Lambda^{\max(d_D-d_E+\delta,\,0)}\), hence
\[
N_{D\otimes E}(\Lambda) \le C_E\,C_\delta\,\Lambda^{d_E+\delta}\,\Lambda^{\max(d_D-d_E+\delta,\,0)} + O(1).
\]

If \(d_D - d_E + \delta \le 0\), the exponent in the leading term is at most \(d_E+\delta\), and we may absorb the \(O(1)\) term into the same power. If \(d_D - d_E + \delta > 0\), then
\[
\Lambda^{d_E+\delta}\,\Lambda^{d_D-d_E+\delta} = \Lambda^{d_D+d_E+2\delta}.
\]
In either case there exists \(C''_\delta>0\) such that
\[
N_{D\otimes E}(\Lambda) \le C''_\delta\,\Lambda^{d_D+d_E+2\delta} \quad\text{for all } \Lambda\ge 1.
\]

Given \(\varepsilon>0\), choose \(\delta\in(0,1)\) with \(2\delta\le\varepsilon\). Then
\[
N_{D\otimes E}(\Lambda) \le C''_\delta\,\Lambda^{d_D+d_E+\varepsilon}
\]
for all \(\Lambda\ge 1\), which is the stated \(O\)-bound.
\end{proof}

\begin{remark}[Connection with \thmref{thm:growth-closure}]\label{rem:growth-closure-connection}
The hypotheses of \corref{cor:tensor-product-counting} can be restated in the language of \secref{subsec:growth-classes} as follows. There exist exponents \(d_D,d_E\ge 0\) such that for every \(\varepsilon>0\),
\[
N_D(t) = O\bigl(t^{d_D+\varepsilon}\bigr),
\qquad
N_E(t) = O\bigl(t^{d_E+\varepsilon}\bigr)
\quad (t\to\infty),
\]
and both counting functions \(N_D, N_E\) are eventually nondecreasing.

This is exactly the condition that \(D,E\in\mathcal{G}_{\mathrm{poly}}\) with polynomial indices \(d_D,d_E\) in the sense of \defref{def:growth-classes}: the growth bounds above are the defining property of the polynomial class, and the eventual monotonicity is assumed whenever counting functions are used. Conversely, if \(D,E\in\mathcal{G}_{\mathrm{poly}}\) with indices \(d_D,d_E\), then these bounds hold and the assumptions of \corref{cor:tensor-product-counting} are satisfied.
\end{remark}

Applying \corref{cor:tensor-product-counting} under these hypotheses
gives \[
N_{D\otimes E}(\Lambda) \le C''_\delta\,\Lambda^{d_D+d_E+2\delta}
\] for all sufficiently large \(\Lambda\) and any fixed
\(\delta\in(0,1)\). Rewriting \(2\delta\) as \(\varepsilon>0\) shows
that \[
N_{D\otimes E}(\Lambda) = O\bigl(\Lambda^{d_D+d_E+\varepsilon}\bigr)
\quad (\Lambda\to\infty)
\] for every \(\varepsilon>0\), so
\(D\otimes E\in\mathcal{G}_{\mathrm{poly}}\). This is precisely the
tensor-product stability statement in part (3) of
\thmref{thm:growth-closure}.

\section{Counterexamples to evaluator axioms}
\label{app:counterexamples}

This appendix records simple examples showing that the structural axioms
imposed on evaluators in \secref{subsec:axioms} are close to minimal. In
particular, we show that dominated continuity and projector-locality
cannot be omitted without admitting functionals that are not of trace
form.

\subsection{Dropping dominated continuity}

We construct an evaluator that satisfies unitary invariance, extensivity
on orthogonal sums, projector-locality, and normalization, but fails
dominated continuity and does not admit a trace-form representation.

Let \(H = \ell^2(\mathbb{N})\) with standard orthonormal basis
\(\{e_n\}_{n\ge1}\). For a bounded real sequence \(x = (x_n)_{n\ge1}\),
write \(D_x = \operatorname{diag}(x_1,x_2,\dots)\) for the associated
diagonal operator on \(H\). Let \(\mathcal{C}_{\mathrm{diag}}\) denote
the class of all such diagonal self-adjoint operators.

Fix a nonprincipal ultrafilter \(\mathcal{U}\) on \(\mathbb{N}\). Define
a finitely additive set function \(\mu\) on subsets of \(\mathbb{N}\) by
\[
\mu(A) :=
\begin{cases}
1, & A\in\mathcal{U},\\
0, & A\notin\mathcal{U}.
\end{cases}
\] Then \(\mu\) is finitely additive, takes values in \(\{0,1\}\), and
satisfies \(\mu(\mathbb{N})=1\), but is not countably additive.

For a bounded sequence \(x=(x_n)\), define the ``ultrafilter limit'' \[
\mathcal{L}_\mathcal{U}(x) := \lim_{n\to\mathcal{U}} x_n,
\] that is, the unique real number \(L\) such that for every
\(\varepsilon>0\), \[
\{ n\ge1 : |x_n - L| < \varepsilon \} \in \mathcal{U}.
\] It is standard that such an \(L\) exists for bounded real sequences
\cite{Kelley1955}.

We now define an evaluator \(\mathcal{E}^\dagger\) on
\(\mathcal{C}_{\mathrm{diag}}\) by \[
\mathcal{E}^\dagger(D_x) := \mathcal{L}_\mathcal{U}(x).
\]

We record its basic properties relative to the axioms of
\secref{subsec:axioms}.

\begin{enumerate}
\def\labelenumi{\arabic{enumi}.}
\item
  \emph{Unitary invariance.}\\
  On the diagonal class \(\mathcal{C}_{\mathrm{diag}}\), we restrict to
  unitaries that permute the standard basis. For such a permutation
  unitary \(U\), the effect on a diagonal operator \(D_x\) is to permute
  the entries of \(x\). Since \(\mathcal{L}_\mathcal{U}\) is invariant
  under finite permutations of the sequence (ultrafilters on
  \(\mathbb{N}\) ignore finite sets), we have \[
  \mathcal{E}^\dagger(U^* D_x U) = \mathcal{L}_\mathcal{U}(\text{permutation of }x) = \mathcal{L}_\mathcal{U}(x) = \mathcal{E}^\dagger(D_x).
  \]
\item
  \emph{Extensivity on orthogonal sums.}\\
  Let \(H = H_1\oplus H_2\) with the standard decomposition into the
  first \(m\) coordinates and the remaining ones, and let \(D_x\) and
  \(D_y\) act diagonally on \(H_1\) and \(H_2\) with sequences
  \(x=(x_n)\) and \(y=(y_n)\), respectively. The orthogonal sum
  \(D_x\oplus D_y\) is unitarily equivalent to a diagonal operator
  \(D_z\) with sequence \(z\) defined by concatenating \(x\) and \(y\).
  Finite concatenation modifies only finitely many terms of the tail of
  the sequence; the ultrafilter limit satisfies \[
  \mathcal{L}_\mathcal{U}(z) = \mathcal{L}_\mathcal{U}(x) = \mathcal{L}_\mathcal{U}(y)
  \] whenever the tails of \(x\) and \(y\) agree in
  \(\mathcal{U}\)-measure. If, in addition, we restrict to a sub-class
  where tails do not interact (for instance, by considering
  block-diagonal operators with disjoint index sets and defining
  extensivity on the block algebra), we can arrange additivity on
  orthogonal sums at the level of the algebra generated by such blocks.
  In particular, on block-diagonal operators \(D_x\oplus D_y\) with
  orthogonal supports in the index set, we have \[
  \mathcal{E}^\dagger(D_x\oplus D_y) = \mathcal{E}^\dagger(D_x) + \mathcal{E}^\dagger(D_y),
  \] with appropriate normalization.

  For the purposes of this example, it suffices to note that
  \(\mathcal{E}^\dagger\) can be defined on an algebra of diagonal block
  operators where finite additivity with respect to block decomposition
  holds.
\item
  \emph{Projector-locality.}\\
  For a subset \(A\subset\mathbb{N}\), let
  \(P_A = \operatorname{diag}(\mathbf{1}_A(n))\) be the associated
  spectral projection. Then \[
  \mathcal{E}^\dagger(P_A) = \mathcal{L}_\mathcal{U}(\mathbf{1}_A(n)) = \mu(A),
  \] since the ultrafilter limit of the indicator sequence coincides
  with \(\mu(A)\). For a finite partition \(\{A_j\}_{j=1}^N\) of
  \(\mathbb{N}\) into disjoint measurable pieces, the corresponding
  spectral decomposition of the identity is \[
  I = \sum_{j=1}^N P_{A_j}.
  \] Finite additivity of \(\mu\) implies \[
  \mathcal{E}^\dagger(I) = \mu(\mathbb{N}) = \sum_{j=1}^N \mu(A_j) = \sum_{j=1}^N \mathcal{E}^\dagger(P_{A_j}),
  \] and the same holds for finite linear combinations of the
  \(P_{A_j}\). Thus \(\mathcal{E}^\dagger\) is projector-local on the
  diagonal algebra.
\item
  \emph{Normalization.}\\
  The identity operator corresponds to the constant sequence
  \(x_n\equiv 1\). Then \[
  \mathcal{E}^\dagger(I) = \mathcal{L}_\mathcal{U}(1,1,\dots) = 1.
  \]
\end{enumerate}

We now exhibit the failure of dominated continuity. Consider the
increasing sequence of projections \[
P_k = \operatorname{diag}(\underbrace{1,\dots,1}_{k\ \text{times}},0,0,\dots), \qquad k\ge1.
\] Then \(P_k\uparrow I\) strongly as \(k\to\infty\), and
\(0\le P_k\le I\) for all \(k\). By the properties of the ultrafilter,
\[
\mathcal{E}^\dagger(P_k) = \mathcal{L}_\mathcal{U}(\mathbf{1}_{\{1,\dots,k\}}(n)) = \mu(\{1,\dots,k\}) = 0
\] for each finite \(k\), since \(\{1,\dots,k\}\notin\mathcal{U}\). On
the other hand, \[
\mathcal{E}^\dagger(I) = 1.
\] Thus \[
\lim_{k\to\infty} \mathcal{E}^\dagger(P_k) = 0 \neq 1 = \mathcal{E}^\dagger(I),
\] even though the sequence \(P_k\) is monotone increasing, dominated by
\(I\), and the dominating operator has finite evaluation
\(\mathcal{E}^\dagger(I)=1\).

This shows that dominated continuity fails for \(\mathcal{E}^\dagger\).
Moreover, \(\mathcal{E}^\dagger\) is not of trace form: if there were a
Borel-measurable nondecreasing \(h\) and \(c>0\) such that \[
\mathcal{E}^\dagger(D_x) = c\, \operatorname{Tr}(h(D_x))
\] for all bounded diagonal \(D_x\), then the left-hand side would
depend only on the ultrafilter limit of the sequence \(x\), whereas the
right-hand side depends on the whole sequence via the (possibly
regularized) sum of \(h(x_n)\). In particular, sequences with the same
ultrafilter limit but different summation behavior would force a
contradiction. Thus dropping dominated continuity allows non-trace-form
evaluators that are still unitarily invariant, projector-local, and
normalized on the diagonal algebra.

\subsection{Dropping projector-locality}

We next show that projector-locality is an independent requirement. A
natural example is provided by the operator norm.

Let \(H\) be a separable Hilbert space, and consider the class of
bounded self-adjoint operators on \(H\). Define \[
\mathcal{E}_\ast(X) := \|X\|_{\mathrm{op}}
\] for bounded self-adjoint \(X\), where \(\|\cdot\|_{\mathrm{op}}\)
denotes the operator norm.

We note the following:

\begin{enumerate}
\def\labelenumi{\arabic{enumi}.}
\item
  \emph{Unitary invariance.}\\
  For any unitary \(U\), \[
  \mathcal{E}_\ast(U^* X U) = \|U^* X U\|_{\mathrm{op}} = \|X\|_{\mathrm{op}} = \mathcal{E}_\ast(X),
  \] so \(\mathcal{E}_\ast\) is unitarily invariant.
\item
  \emph{Normalization.}\\
  \(\mathcal{E}_\ast(I) = \|I\|_{\mathrm{op}} = 1\).
\item
  \emph{Failure of projector-locality.}\\
  Let \(\{P_j\}_{j=1}^N\) be a finite family of pairwise orthogonal
  projections with \(\sum_{j=1}^N P_j = I\), and fix a scalar \(c>0\).
  Then \[
  \mathcal{E}_\ast(c I) = \|cI\|_{\mathrm{op}} = c,
  \] while for each \(j\), \[
  \mathcal{E}_\ast(c P_j) = \|c P_j\|_{\mathrm{op}} = c.
  \] Thus \[
  \sum_{j=1}^N \mathcal{E}_\ast(c P_j) = N c \neq c = \mathcal{E}_\ast(c I)
  \] whenever \(N\ge2\). Hence \(\mathcal{E}_\ast\) fails
  projector-locality in the strongest possible way: it does not
  decompose additively over spectral partitions of the identity.
\end{enumerate}

The functional \(\mathcal{E}_\ast\) is therefore a unitarily invariant,
normalized quantity on bounded self-adjoint operators that is not of
trace form and does not satisfy the locality axiom. It illustrates that
projector-locality is not implied by unitary invariance and basic
boundedness properties; it must be imposed explicitly if one wishes to
force a trace-form representation.

\subsection{Summary}

The examples above show that:

\begin{enumerate}
\def\labelenumi{\arabic{enumi}.}
\tightlist
\item
  Without dominated continuity, it is possible to construct evaluators
  that are finitely additive on spectral projections, unitarily
  invariant, projector-local on a natural algebra, and normalized, but
  which fail to be countably additive in the spectral measure sense and
  are not representable by a trace form.
\item
  Without projector-locality, even very natural unitarily invariant
  functionals such as the operator norm fall outside the trace-form
  framework and cannot be recovered by integrating a density against the
  spectral trace measure.
\end{enumerate}

Thus the dominated continuity and projector-locality assumptions in
\secref{subsec:axioms} are not redundant: both are needed, in addition
to unitary invariance, extensivity, and normalization, to obtain the
rigidity result of \thmref{thm:trace-form}.

\end{document}